%% file: fiberedlink.tex
\documentclass[a4paper,11pt,twoside,openright,notitlepage]{article}

\usepackage[british,english]{babel}    

\usepackage[utf8]{inputenc}    

\usepackage{graphicx}            

\usepackage{amsmath,amssymb,amsthm,amscd}   

\usepackage{mathrsfs}

\usepackage[binding=5mm]{layaureo}        

\usepackage{indentfirst}        

\usepackage{makeidx}	

\usepackage{fancyhdr}

\usepackage[all]{xy}

\usepackage{stmaryrd}

\usepackage{matematica}

\usepackage{color} 

\title{Fibered knots and links in lens spaces}

\author{Enrico Manfredi and Alessio Savini}

\date{31 January 2014}

\begin{document}

\maketitle

\begin{abstract}
We take advantage of the correspondence between fibered links, open book decompositions and contact structures on a closed connected 3-dimensional manifold to determine a mixed link diagram presentation for a particular fibered link $L$ in the lens space $L(p,q)$.
Moreover, we construct a diagram for the lift of $L$ in the 3-sphere $S^3$ in order to show that, for $q=1$, $L$ is compatible with the standard contact structure of $L(p,1)$ when $p \leq -1$.

\let\thefootnote\relax

\end{abstract}

\vspace{20pt}


\input{materialeiniziale/Introduzione}

\input{sezioni/sezione1}
\input{sezioni/sezione2}
\input{sezioni/sezione3}
\input{sezioni/sezione4}


\input{materialefinale/Bibliografia}

\end{document}

%% file: materialeiniziale/Introduzione.tex
\section{Introduction}

In~\cite{harer:articolo} Harer defines two different operations which one can apply to a fixed fibered knot/link in a 3-dimensional manifold $M^3$ in order to get all the other possible examples of fibered knot/link. A knot $K$ is a fibered knot if there exists a map $\pi: M^3 \setminus K \rightarrow S^1$ which is a fiber bundle and a homeomorphism $\varphi: K \times D^2 \rightarrow \nu(K)$, where $\nu(K)$ is a suitable tubular neighborhood $K$ such that $\varphi \circ \pi|_{K \times \partial D^2}$ coincides with the projection on the second factor. For instance, the trivial knot in the 3-sphere $S^3$ satisfies these properties.\\ 
The interest in the study of fibered knots can be justified by several reasons. In~\cite{skora:articolo} Skora assumes to fix a fibered knot as braid axis in order to extend the Alexander trick and to generalize Markov's theorem to a generic 3-manifold.
Another remarkable aspect is the possibility to translate topological results in terms of differential properties. More precisely, in~\cite{giroux:articolo} Giroux establishes a one to one correspondence between the set of open book decompositions of a 3-manifold, a notion completely equivalent to the one of fibered link, and the set of positive contact structures on the same manifold.\\
The main goal of this paper is to construct a representation via mixed link diagram of a particular fibered link $L$ in the lens space $L(p,q)$. Starting from $L$ it is possible to construct all the possible fibered knots/links in $L(p,q)$ by applying a suitable sequence of Harer moves. In literature the link $L$ is already known for specific values of $(p,q)$, but it is usually represented as its corresponding abstract open book decomposition (see for example~\cite{pavelescu:articolo}).\\
The structure of the paper is as follows. In Section 2 we 
 fix the notation regarding knot theory and 3-dimensional topology. After recalling the definition of link in a 3-manifold, we give a brief description of surgery theory, which allows to define Kirby diagrams for a framed link and Kirby moves. Next, our attention is focused on the lens space $L(p,q)$. After some equivalent definitions of it, we introduce two ways to represent links embedded in this particular space, the mixed link diagram and the band diagram.
\\
In Section 3 are exposed the notions of fibered link, open book decomposition and contact structure, all of them relative to the same 3-manifold $M^3$. In this section we recall Harer moves and we give a precise description of the correspondence between the set of fibered links up to positive plumbing, the set of open book decompositions up to positive stabilization and the set of contact structures up to isotopy.
\\
We start Section 4 with Lemma \ref{obdeffects}, which describes the effects produced on a given open book by performing a surgery transversal or parallel to each page with framing 0 or 1, respectively. This lemma turns out to be essential for our purposes. Indeed, if we apply it, together with Kirby moves, to a $p$-framed surgery along the unknot which is transversal to each page of the standard open book decomposition $(D^2,\text{Id})$ of $S^3$, we get back the desired presentation of the fibered link $L$ in $L(p,1)$, as stated in Proposition \ref{casep1}. The same procedure can be generalized by taking the integral presentation $L(p,q)$, as in Proposition \ref{fiberedlinkprop}.
\\
In the last section we outline a procedure which allows us to obtain the lift in $S^3$ of the fibered link $L$ in $L(p,1)$, in order to verify that $L$ is compatible with the contact structure induced by quotienting the standard contact structure on the 3-sphere $S^3$ under the $\mathbb{Z}_p$-action when $p<-1$.\\

%% file: sezioni/sezione1.tex
\section{Representation of links in lens spaces}

\subsection{Basic definitions}

Let \textit{Top} be the category of topological manifolds and continuous maps and let \textit{Diff} be the category of differentiable manifolds and smooth functions. Since in~\cite{kirby:libro} it is proved that these categories are equivalent in dimension three, we will provide a differentiable structure on a topological 3-manifold each time it will be necessary.\\
Given a closed orientable 3-dimensional manifold $M^3$, we define a \textit{link} $L$ in $M^3$ as a finite collection of smooth embeddings $\gamma_i : S^1 \rightarrow M^3$ whose images $L_i$ are pairwise disjoint. Each curve is said to be a \textit{component} of the link. A link with only one component is called a \textit{knot}. A link $L \subset M^3$ is \textit{trivial} if each component bounds a disk in $M^3$ in such a way that each disk is disjoint from the others.
\\
Two different links $L_1$ and $L_2$ in $M^3$ are said to be \textit{homeo-equivalent} if there \mbox{exists} a homeomorphism $F:M^3 \rightarrow M^3$ such that $F(L_1)=L_2$. In an analogous way, we will say that $L_1$ and $L_2$ are \textit{isotopy-equivalent} if there exists an isotopy $H: M^3 \times [0,1] \rightarrow M^3$ such that $h_0=\text{id}_{M^3}$ and $h_1(L_1)=L_2$, where $h_t$ is defined as $h_t(m):=H(m,t)$ with $m \in M^3$. Clearly these two equivalence relations are \mbox{different} so, to avoid any ambiguity, we choose to consider links up to homeo-equivalence until the end of the paper.
\\
Thanks to the use of the previous definitions, it is now possible to describe a procedure which allows us to construct new 3-manifolds starting from a fixed knot or link in $M^3$. Let $K$ be a knot in a smooth closed orientable 3-dimensional manifold $M^3$. We observe that it is possible to thicken $K$ to obtain its closed tubular neighborhood $\nu(K)$. By cutting along the boundary $\partial \nu(K)$, we get two different 3-manifolds: the first one is the \textit{knot exterior}, i.e. the closure of $M^3 \setminus \nu(K)$, and the other one is the solid torus $\nu(K)$, which can be identified with the standard solid torus $S^1 \times D^2 $. Now we can choose an arbitrary homeomorphism $h: \partial \nu (K) \rightarrow \partial(\overline{M^3 \setminus \nu(K)})$ to sew the solid torus back. In this way, we will obtain a new \mbox{manifold} $\tilde M ^3:= \nu(K) \cup_h (\overline{M^3\setminus \nu(K)})$. We say that the manifold $\tilde M^3$ is obtained from $M^3$ by a \textit{a surgery operation \mbox{along $K$}}.
\\
The construction of the new manifold $\tilde M^3$ depends on the choice of the homeomorphism $h$. More precisely, $\tilde M^3$ is completely determined by the image of the meridian $\partial D^2 \times \{\ast\}$ of the solid torus $\nu(K)$, that is by the curve $\gamma=h(\partial D^2 \times \{\ast\})$. 
\\
Let us restrict our attention to the case $M^3=S^3$. If $K$ is a knot in $S^3$, the homology of the complement $V:=\overline{S^3 \setminus \nu(K)}$ is given by $H_0(V)=H_1(V) \cong \mathbb{Z}$ and $H_i(V) \cong 0$, for $i \geq 2$. A generator of the group $H_1(V)$ can be represented by an essential curve $\alpha$ lying on the torus $\partial V$. This curve is said to be a \textit{meridian} of $K$. In the same way, by picking up a curve $\beta$ lying on $\partial V$ which is nullhomologous in $V$ but not in $\partial V$, we get a curve called \textit{longitude} (sometimes is also called \textit{canonical longitude}). 

\begin{figure}[!h]
	\centering
		\includegraphics[width=9cm]{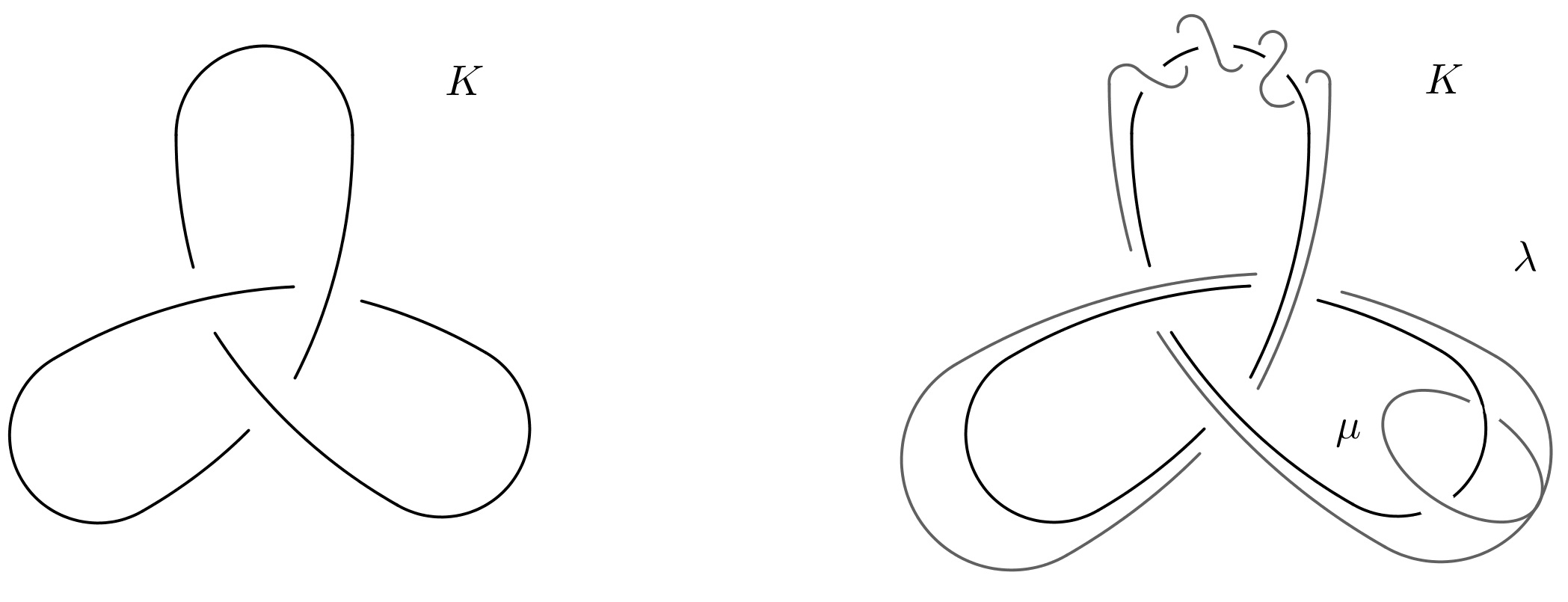}
	\caption{Meridian and parallel for the trefoil knot.}
\end{figure}

If we now consider an essential curve $\gamma$ on the boundary $\partial V$, this curve can always be written in the form $\gamma = p \alpha + q \beta$, with $(p,q)$ pair of coprime integers, after a suitable isotopy. Moreover, the pairs $(p,q)$ and $(-p,-q)$ refer to the same curve, since the orientation of $\gamma$ is not influent. Thinking of a pair $(p,q)$ as a reduced fraction $p/ q$ establishes a one-to-one correspondence between the isotopy classes of non trivial closed curves on $\partial V$ except for $\{\alpha\}$ and the set of rational numbers $\mathbb{Q}$. If we want to include $\alpha$, we must extend $\mathbb{Q}$ with another element indicated by $1/0=\infty$. We write $\bar{\mathbb{Q}} := \mathbb{Q} \cup \{ \infty \}$.
\\
A surgery described by a pair $(p,q)$ of coprime integers (o equivalently by an element of $\bar{\mathbb{Q}}$) is called \textit{rational}. The surgery is called \textit{integral} if $q = \pm 1$.
Similarly, one can define rational and integral surgeries along a link $L \subset S^3$: the surgery coefficient along each component should be rational, respectively, integral.
\\
A good way to represent surgeries along a link, both rational and integral, is to draw the link and write on each component the corresponding reduced fraction. We report two examples in Figure~\ref{framedlink}.

\begin{figure}[!h]
	\centering
		\includegraphics[width=6cm]{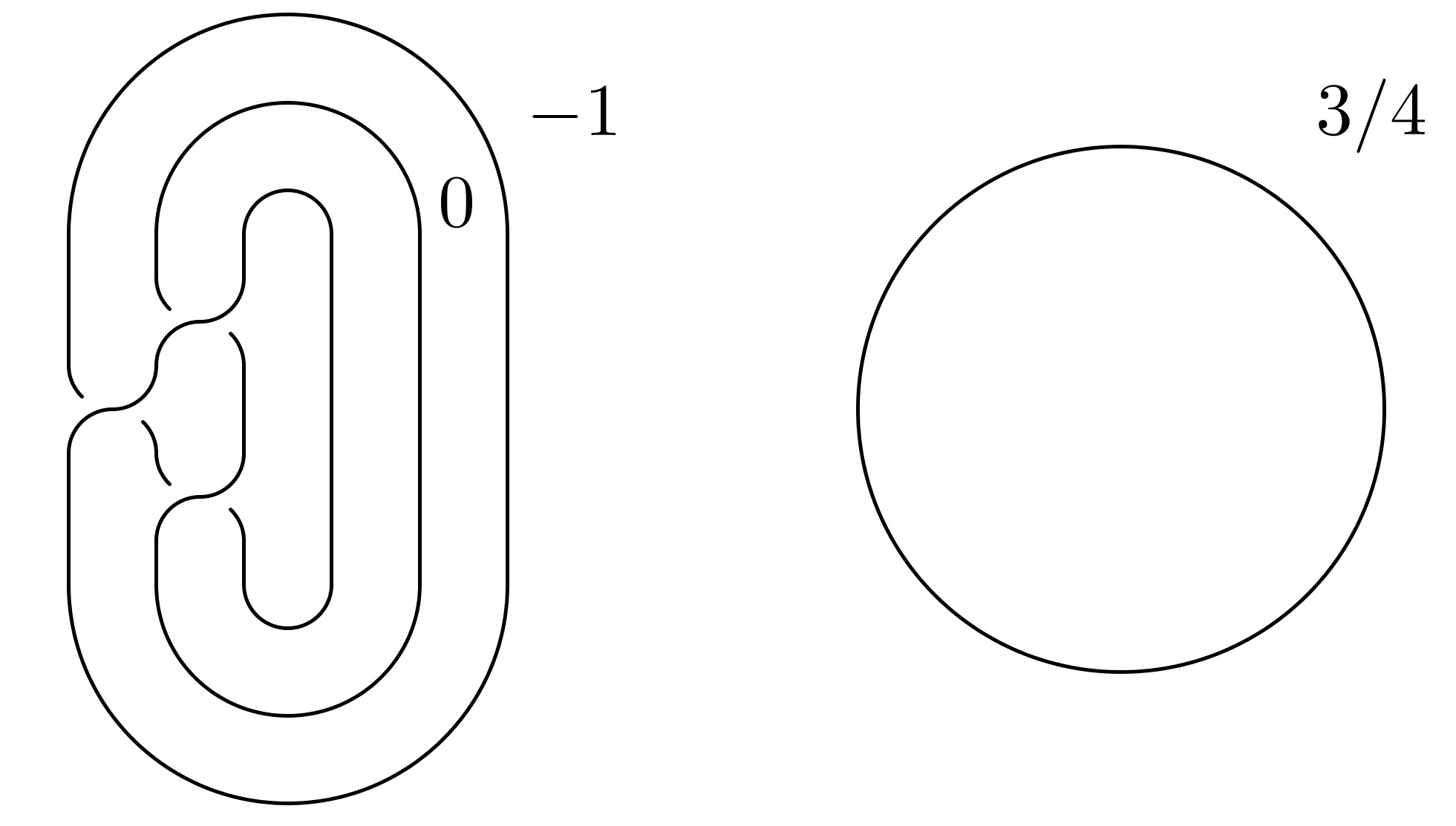}
	\caption{Examples of surgery diagrams.}\label{framedlink}
\end{figure}

The assumption made so far of considering the 3-sphere as ambient manifold for surgery operations reveals not to be restrictive since the following theorem holds (we refer to Theorem 3.1 of~\cite{saveliev:libro} for details).

\begin{teor}[{\upshape ~\cite{lickorish:articolo},~\cite{wallace:articolo}}] \label{lickorish}
Every closed orientable 3-manifold  can be obtained from $S^3$ by an integral surgery on a link $L \subset S^3$.
\end{teor}

Thanks to Theorem \ref{lickorish}, from now until the end of the section we will consider only integral surgeries on links in $S^3$. We refer to~\cite{saveliev:libro} for the part that follows. First of all, we notice that an integral surgery along a link $L$ is equivalent to the choice of an integer for each component. This choice is usually called a \textit{framing} for $L$ and $L$ is called \textit{framed link}.
\\
From Theorem \ref{lickorish} a question could raise quite naturally: given two different framed links, how can we understand if the resulting 3-manifolds are homeomorphic?
In order to answer to this question, we need to introduce two operations, called \textit{Kirby moves}, which do not change the 3-manifold represented by a framed link $L$.
\begin{itemize}
	\item \textbf{Move K1}:
	\\
	We add or delete an unknotted circle with framing $\pm 1$ which does not link any other component of $L$ (see Figure~\ref{movek1}).
	\begin{figure}[!h]
 		\centering
			\includegraphics[width=4.5cm]{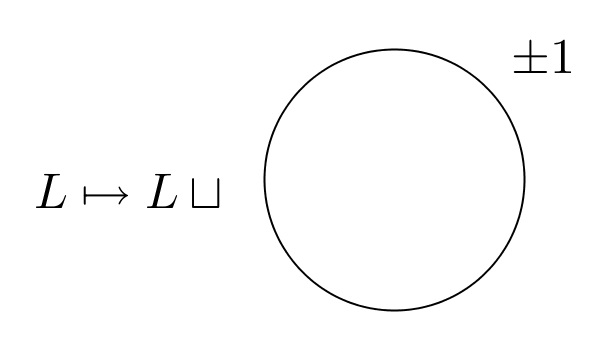}
		\caption{Move K1.}\label{movek1}
	\end{figure}
	
	\item \textbf{Move K2}
	\\
	Let us take $L_1$ and $L_2$ two components of $L$ with framing $n_1$ and $n_2$. Moreover, let $L'_2$ be the curve 	which 	describes the surgery along $L_2$. Now, we substitute the pair $L_1 \cup L_2$ with $L_1 \cup L^\#$, where $L^\#=L_1 \#_b L'_2$ is the \textit{sum along a band} $b$ connecting $L_1$ to $L'_2$ and disjoint from the other components, which remain unchanged after this substitution. The sum $L_1 \#_b  L_2'$ is obtained as follows. We start considering a band $b=[-1,1] \times [-\epsilon, \epsilon]$ with $\{-1\} \times [-\epsilon, \epsilon]$ lying on $L_1$ and $\{1\} \times [-\epsilon, \epsilon]$ lying on $L_2'$. Now, to get $L^\#$, we substitute $\{ \pm 1 \} \times [-\epsilon, \epsilon]$ with the sides of the band corresponding to $[-1,1] \times \{-\epsilon, \epsilon\}$ (see Figure \ref{movek2}).
All components but $L_1$ do not change their framing. The framing of the new component $L^\#$ is given by the formula
		\[
		n^\#:=n_1 + n_2 + 2 \texttt{lk}(L_1,L_2).
		\]
	The components $L_1$ and $L_2$ have to be oriented in such a way to define a coherent orientation on $L^\#$. The orientation depends also on the choice of the band used in the gluing process.
\begin{figure}[!h]
		\centering
			\includegraphics[width=10cm]{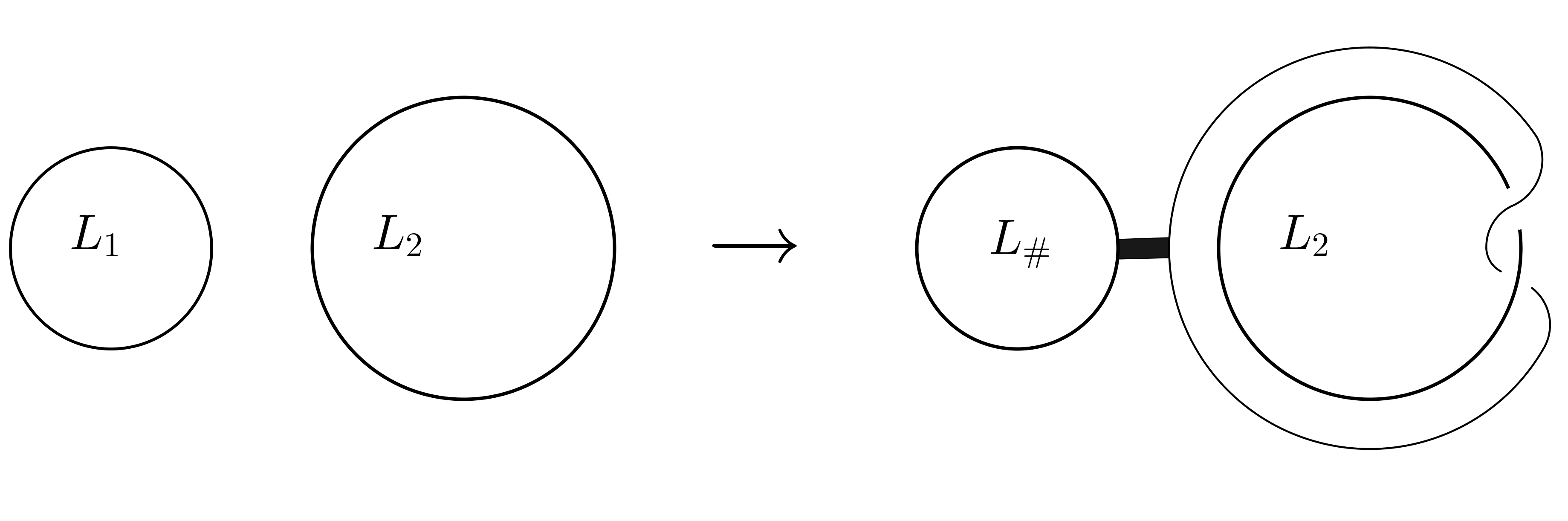}
		\caption{Move K2.}\label{movek2}
	\end{figure}
\end{itemize}


With the introduction of these two moves, we are ready to state Kirby's theorem, which allows us to understand under which conditions two different framed links $L$ and $L'$ determine homeomorphic 3-manifolds.

\begin{teor}[{\upshape ~\cite{kirby:articolo}}]
The closed oriented manifolds obtained by surgery on two different framed links $L$ and $L'$ are homeomorphic by an orientation \mbox{preserving} homeomorphism if and only if the link $L$ can be obtained from $L'$ using a sequence of Kirby moves $K1$ and $K2$.
\end{teor}


\begin{cor}\label{kirbycon}
An unknot with framing $\pm 1$ can always be removed from a framed link L with the effect of giving to all arcs that intersect the disk bounded by the unknot a full left/right twist and changing their framings by adding $\mp 1$ to each arc, assuming they represent different components of L.
\end{cor}

\begin{proof}
We have to repeat the application of move $K2$ to each framed component. Finally we use move $K1$ to remove the unknot framed with $\pm 1$.
\end{proof}

Let us suppose to have a framed link $L$. A direct consequence of the previous corollary is the possibility to modify the framing of a component $K$ of $L$ by introducing or removing a new circle framed by $\pm 1$ which links only $K$ (see Figure~\ref{changeframe}).
\begin{figure}[!h]
	\centering
		\includegraphics[width=14 cm]{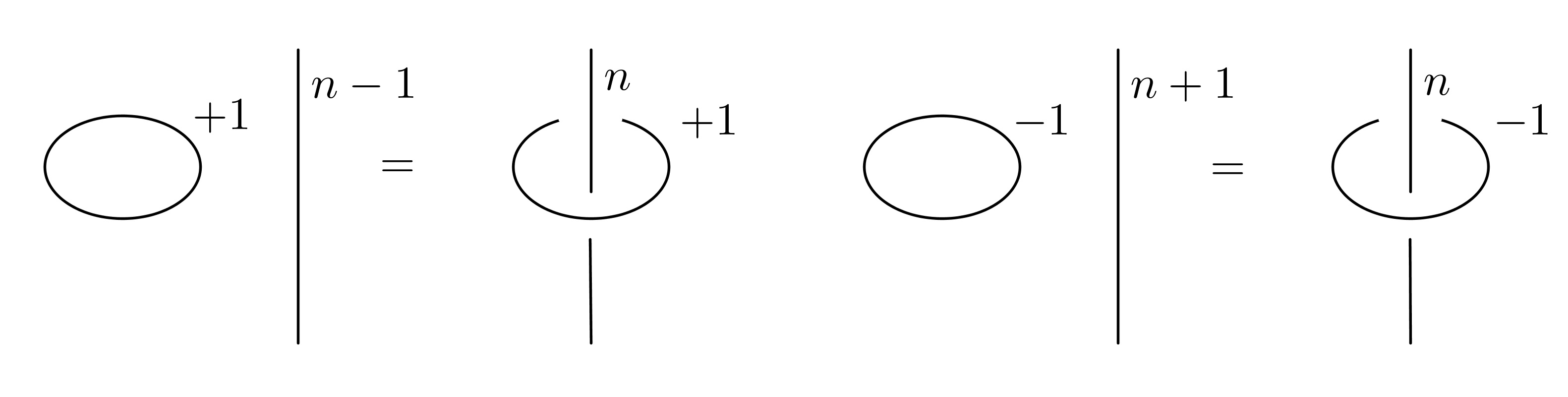}
	\caption{Changing the framing.}\label{changeframe}
\end{figure}
\\
In the same way, Corollary~\ref{kirbycon} suggests that we can eliminate a linking between two components of $L$ by introducing a $\pm 1$-framed circle (see Figure~\ref{nolink}).
\begin{figure}[!h]		
	\centering
		\includegraphics[width=14 cm]{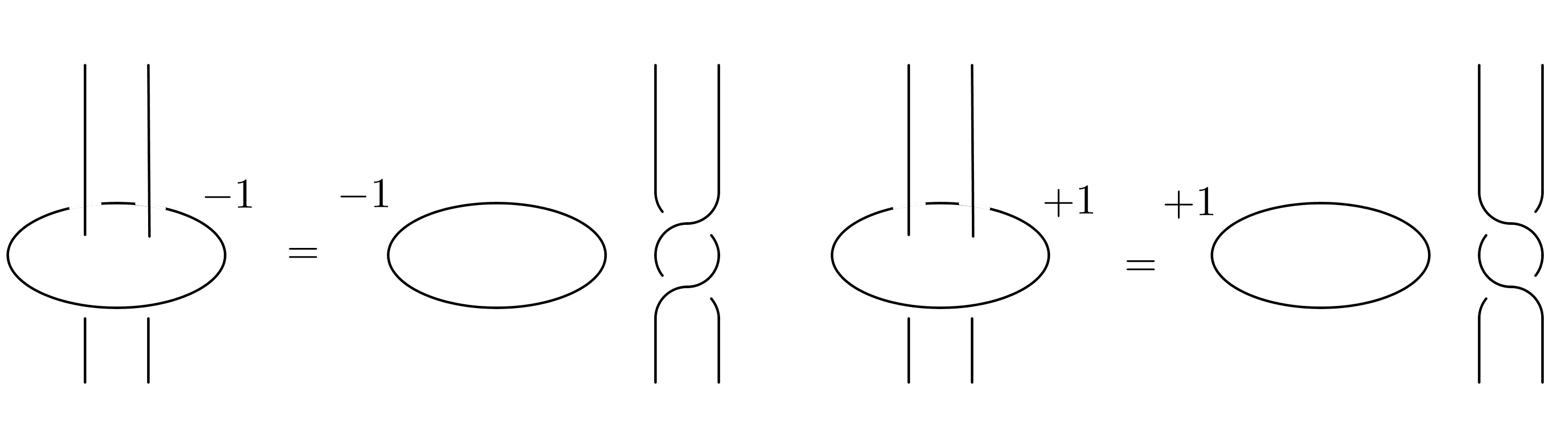}
	\caption{Eliminate a linking between two components.}\label{nolink}
\end{figure}

\subsection{Lens spaces}
A direct way to describe the space $L(p,q)$ is to think of it as the result of a rational surgery along the unknot with coefficient $-p/q$, as shown in Figure~\ref{pqsurgery}.
By assuming the previous definition for the space $L(p,q)$, it is quite immediate to understand that $L(p,q)$ and $L(-p,-q)$ are homeomorphic. For our purposes, we will not actually use this presentation, but we will prefer the integral one reported in Proposition~\ref{integralprop} (see{\upshape ~\cite[Proposition 17.3]{prasolov:libro}}).
\begin{figure}[!h]
	\centering
		\includegraphics[width=3.5cm]{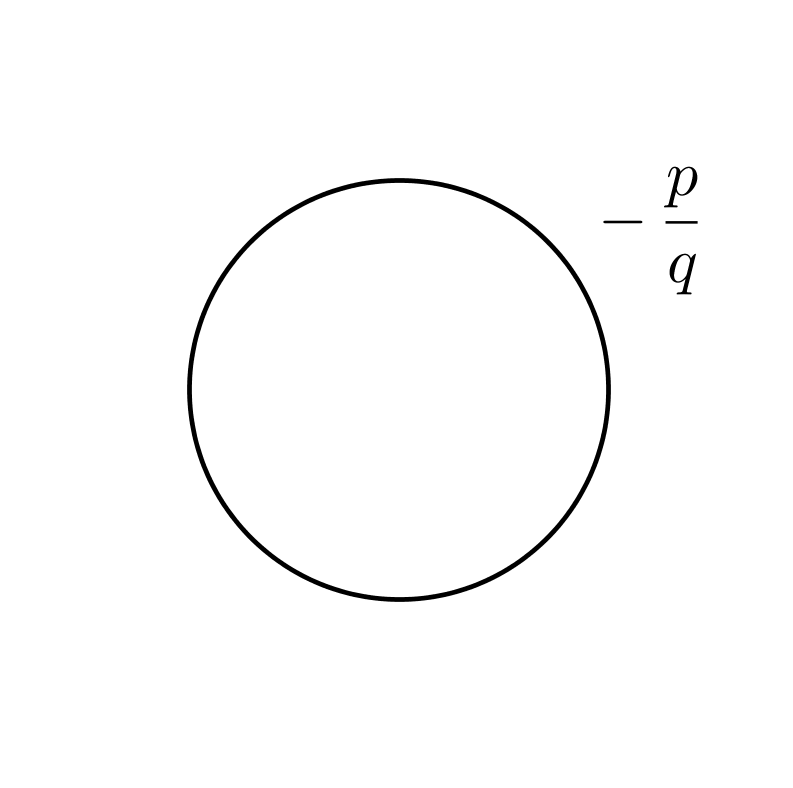}
	\caption{Rational $p/q$-surgery along the unknot.}\label{pqsurgery}
\end{figure}

\begin{prop}\label{integralprop}
Any lens space $L(p,q)$ has a surgery description given by Figure \ref{integralpq}, where $-p/q=[a_1,\ldots,a_n]$ is a continued fraction decomposition, i.e.
\[
[a_1,\ldots,a_{n}]=a_1 - \frac{1}{a_2-\frac{1}{\ldots-\frac{1}{a_{n}}}}, \hspace{15pt} a_i \in \mathbb{Z}
\]
\end{prop}

\begin{figure}[!h]
	\centering
		\includegraphics[width=7cm]{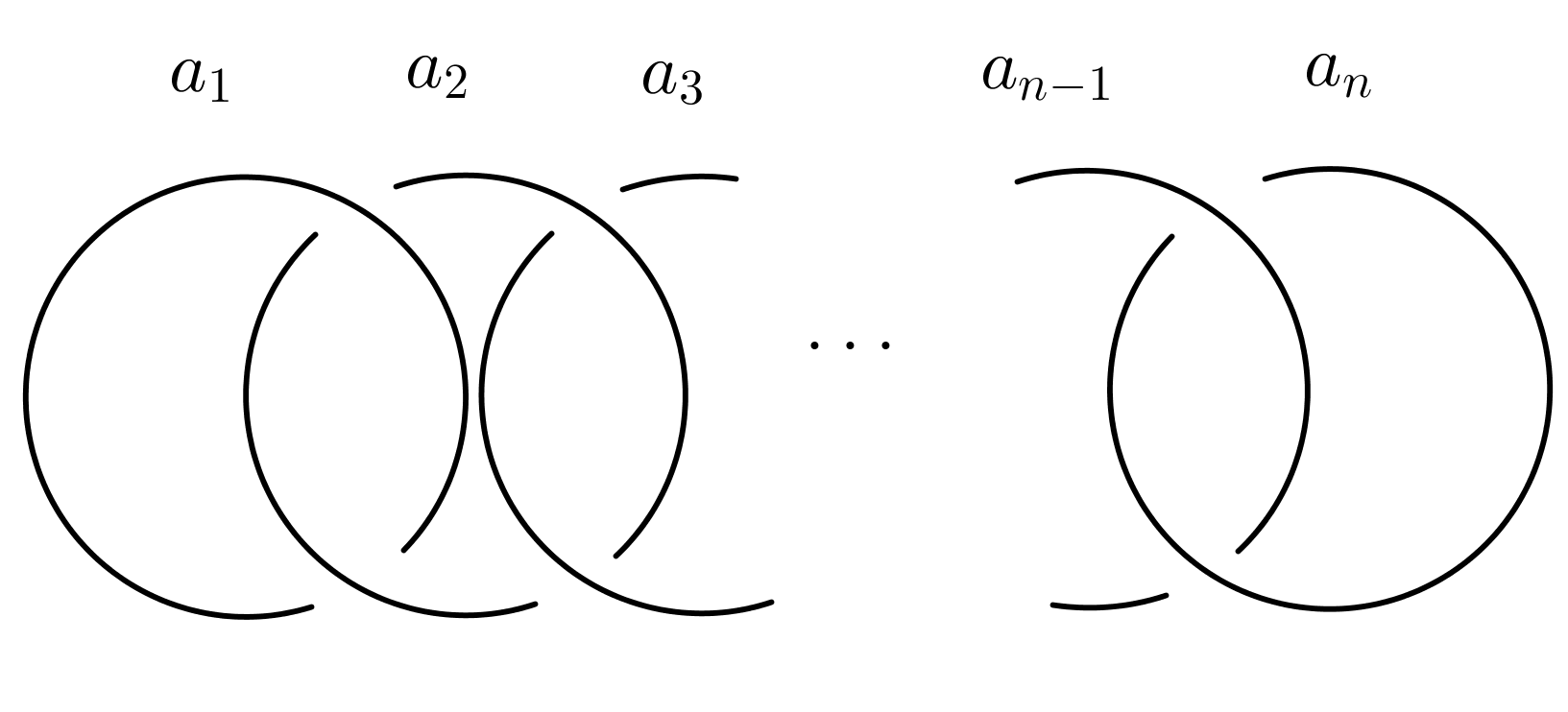}
	\caption{Integral surgery for $L(p,q)$.}\label{integralpq}
\end{figure}

In the case $p \neq 0$, the space $L(p,q)$ has another useful model. It can be thought of as the quotient of $S^3$ under a properly discontinuous action of the group $\mathbb{Z}_p$. To see this, let us consider the 3-sphere as a subspace of $\mathbb{C}^2$, that is 
\[
S^3:= \{ (z,w) \in \mathbb{C} : |z|^2 + |w|^2 =1 \}.
\]

We define on it an action of the cyclic group $\mathbb{Z}_p$ as follows: let us fix a pair of coprime integer $(p,q)$ satisfying $0 \leq q < p$ and put
\[
\cdot: \mathbb{Z}_p \times S^3 \rightarrow S^3, \hspace{15pt} \bar{n} \cdot (z,w):=(e^{\frac{2\pi in}{p}}z,e^{\frac{2\pi i nq}{p}}w).
\]

This map reveals to be a free action of the group on the 3-sphere, indeed
\begin{enumerate}
	\item $\bar{0} \cdot (z,w)=(e^0z,e^0w)=(z,w)$
	\item $\bar{m} \cdot ( \bar{n} \cdot (z,w))=\bar{m} \cdot (e^{\frac{2\pi in}{p}}z,e^{\frac{2\pi i nq}{p}}w)=(e^{\frac{2\pi i(n+m)}{p}}z,e^{\frac{2\pi i (n+m)q}						{p}}w)=(\bar{n}+ \bar{m})(z,w)$
	\item $\bar{n}(z,w)=(z,w)$ implies $e^{\frac{2 \pi i n}{p}}z=z$ thus $z=0$ or $e^{\frac{2 \pi i n}{p}}=1$. In both cases we reduce to $\bar{n}=0$.
\end{enumerate}

Moreover, each element $\bar{n} \in \mathbb{Z}_p$ acts as a homeomorphism or, equivalently, we have a representation $\rho: \mathbb{Z}_p \rightarrow \text{Homeo}(S^3)$.
\\
It can be shown that we may define the quotient space $S^3/\mathbb{Z}_p$ as the lens space $L(p,q)$ for $p \neq 0$. Moreover, from this particular presentation we easily argue that $L(p,q)$ and $L(p,q+np)$ are homeomorphic. This remark and the homeomorphism $L(p,q) \cong L(-p,-q)$ justify the assumption generally made in literature $0 \leq q <p$ whenever $p>0$.

\subsection{Links in lens spaces}

So far we have understood how to represent a 3-dimensional manifold via framed link and when two different framed links return homeomorphic 3-manifolds. The next step is to get a suitable way to visualize links in a general closed orientable 3-manifold $M^3$. To do this, we will use Theorem \ref{lickorish} to represent $M^3$ as an integral surgery along a link (or a knot) in $S^3$.

\begin{deft}
Let $M^3$ be a closed orientable 3-manifold which is obtained by an integral surgery along a framed link $I$ in $S^3$. An link $L$ in $M^3$ is represented as a \textit{mixed link} in $S^3$ if it is given a diagram of the link $I \cup L$ in $S^3$. The link $I$ is called \textit{fixed part}, while the link $L$ is called \textit{moving part} (see Figure~\ref{mixedlink}).
\end{deft}

\begin{figure}[!h]
	\centering
		\includegraphics[width=7cm]{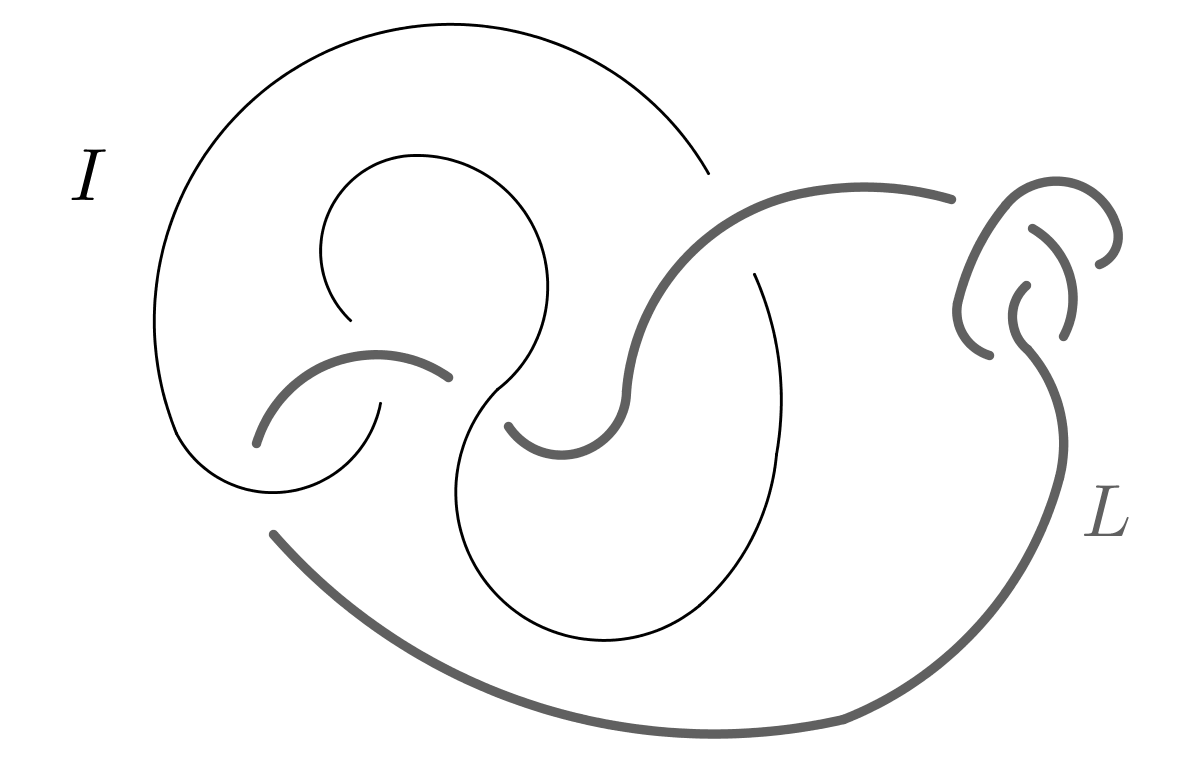}
	\caption{Example of mixed link.}\label{mixedlink}
\end{figure}

Since lens spaces are the result of a $-p/q$-surgery along an unknotted circle in $S^3$, we can try to extend the notion of mixed link diagrams by admitting rational surgery coefficients. With this assumption, a link $L$ in a lens space is representable as a mixed link whose fixed part coincides with the unknot and the surgery coefficient may be rational.
\\
Other two representations of links in lens spaces will be useful for our purposes, namely the punctured disk diagram introduced by~\cite{gabrovsek:articolo} and the closely related band diagram introduced by~\cite{turaev:articolo}. Their definitions can be recovered from a mixed link diagram representation as follows. By denoting the fixed part by $U$ and recalling that $S^3 = \mathbb{R}^3 \cup \{ \infty \}$, we can pick up a point $p$ from $U$ and suppose to identify it with $\infty$. By considering on $\mathbb{R}^3$ the standard coordinates system $(x,y,z)$, we can assume that $U$ is described by the $z$-axis and then we project $U \cup L$ to the $xy$-plane by the map $\pi: \mathbb{R}^3 \rightarrow \mathbb{R}^2$, $\pi(x,y,z)=(x,y)$ (see Figure~\ref{linklens}). We need to suppose that the projection map is \textit{regular}, that is $\pi$ must satisfies the following conditions:
\begin{enumerate}
\item the differential $d_m(\pi|_L)$ of the restriction $\pi|_L$ must have \mbox{rank 1} for every point $m \in L$;
\item no more than two distinct points of the link are projected on one and the same point on the plane;
\item the set of double points, i.e. those on which two points project, is finite and at each double point the projection of the two tangents do not coincide.
\end{enumerate}

\begin{deft}\label{pdiskdiagram}
A \textit{punctured disk diagram} of a link $L$ in $L(p,q)$ is a regular projection of $L \cup U$, where $U$ is described by a single dot on the plane.
\end{deft}

\begin{figure}[!h]
	\centering
		\includegraphics[width = 8.5cm]{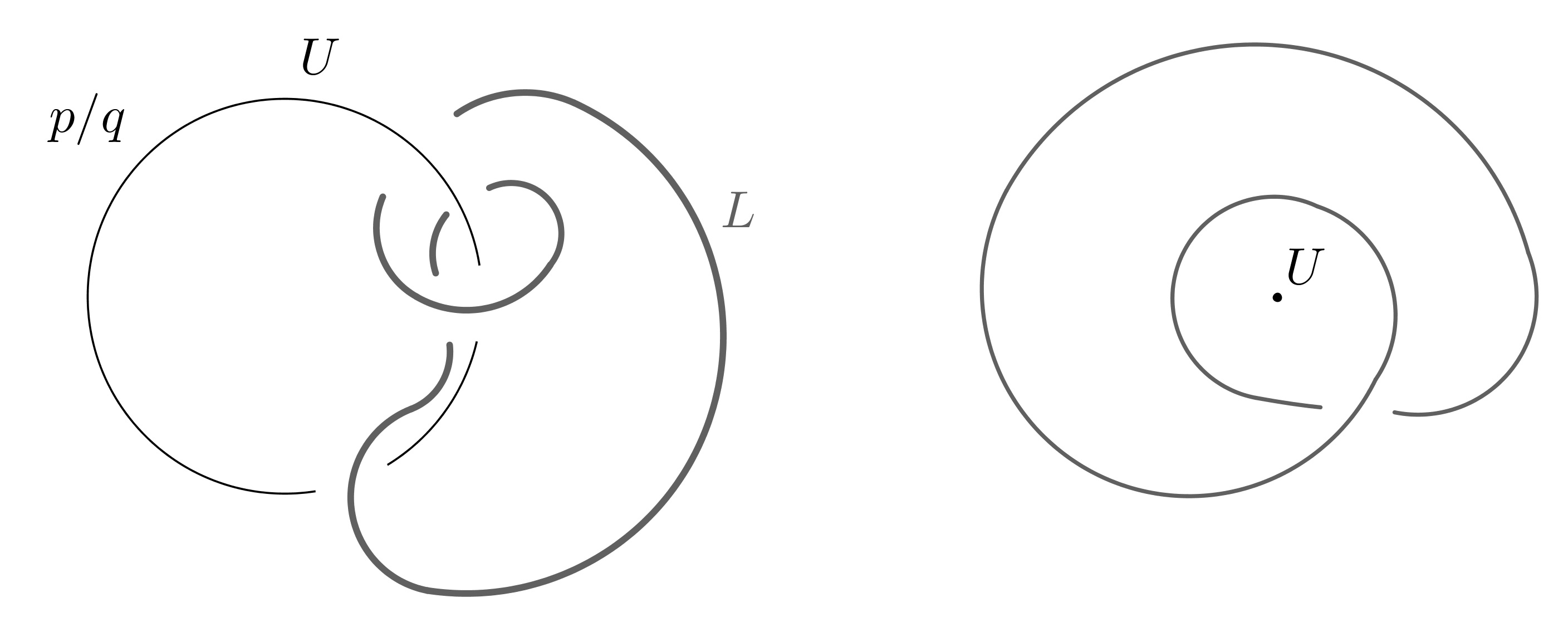}
	\caption{Link in $L(p,q)$: mixed link and punctured disk diagram.} \label{linklens}
\end{figure}

Supposing that the diagram is all contained inside a disk, we remove a neighborhood of the dot representing $U$ in such a way that the diagram of $L$ lies in an annulus. Then we cut along a line orthogonal to the boundary of the annulus, being careful to avoid the crossings of $L$. Finally, we deform the annulus to get a rectangle (Figure \ref{toricband}).

\begin{deft}\label{bandiagram}
The described procedure gives back a presentation of a link in $L(p,q)$ called \textit{band diagram}.
\end{deft}

\begin{figure}[!h]
	\centering
		\includegraphics[width = 8.5cm]{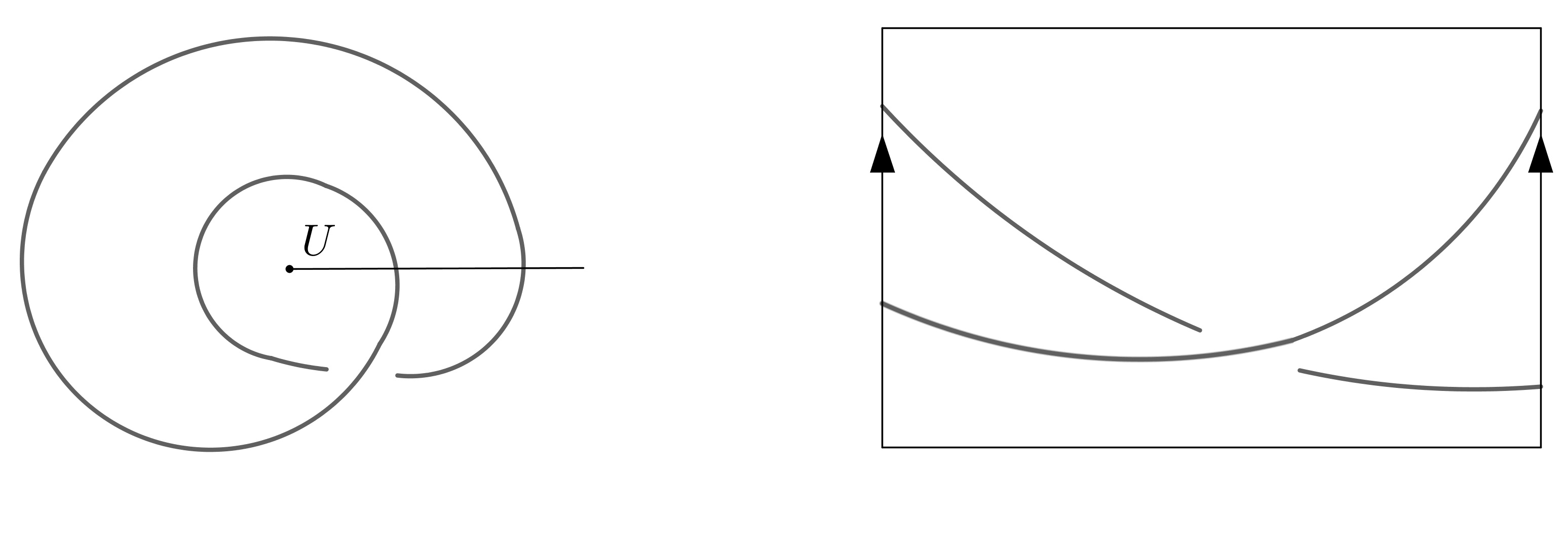}
	\caption{Link in $L(p,q)$: punctured disk diagram and band diagram.}\label{toricband}
\end{figure}

Since a link contained in the 3-sphere may be easier to study rather than a link in $L(p,q)$, we might ask quite naturally how to get a \textit{lift} in the 3-sphere of a link in $L(p,q)$. If $\omega_{p,q}:S^3 \rightarrow L(p,q)$ is the covering map with respect to the action of $\mathbb{Z}_p$ and if $L$ is a link in $L(p,q)$, we say that a link $L'$ contained in the 3-sphere is a lift of $L$ if it coincides with the preimage $L'=\omega_{p,q}^{-1}(L)$. Given a band diagram of a link in $L(p,q)$, the following proposition gives back a useful method to construct its lift in $S^3$. We define the Garnside braid $\Delta_n$ in $n$ strands as 

\[
\Delta_n:=(\sigma_{n-1}\sigma_{n-2}\ldots\sigma_{1})(\sigma_{n-2}\ldots\sigma_1)\ldots\sigma_1,
\]
\begin{figure}[!h]
	\centering
		\includegraphics[width=8.5cm]{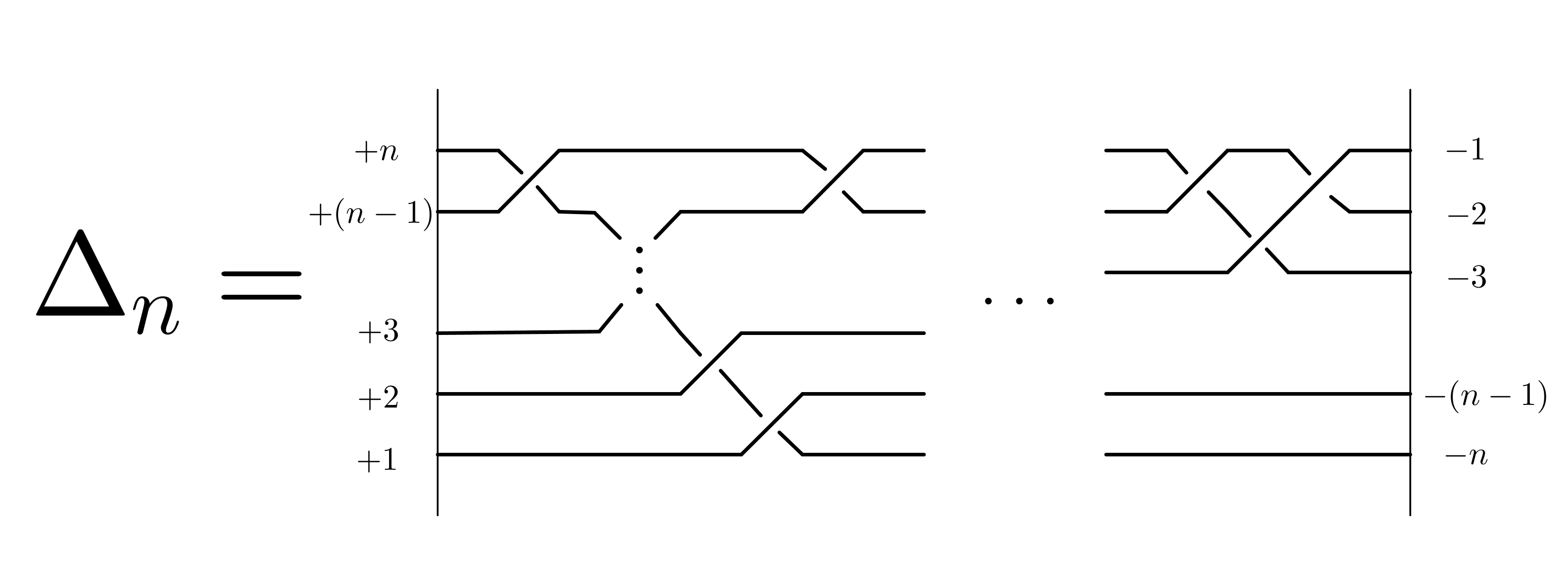}
	\caption{Garnside braid $\Delta_n$}
\end{figure}

where $\sigma_i$ are the Artin's generators of the group of braids in $n$-strands $B_n$.

\begin{prop}[{\upshape ~\cite[Proposition 6.4]{manfredi:tesi}}]\label{lift}
Let $L$ be a link in the lens space $L(p,q)$, with $0 \leq q < p$, and let $B_L$ be a band diagram for L with $n$ boundary points. Then a diagram for the lift $L'$ in the 3-sphere $S^3$ can be found by juxtaposing $p$ copies of $B_L$ and closing them with the braid $\Delta_n^{2q}$ (see Figure~\ref{linkliftprop}).
\end{prop}

\begin{figure}[!h]
	\centering
		\includegraphics[width=8cm]{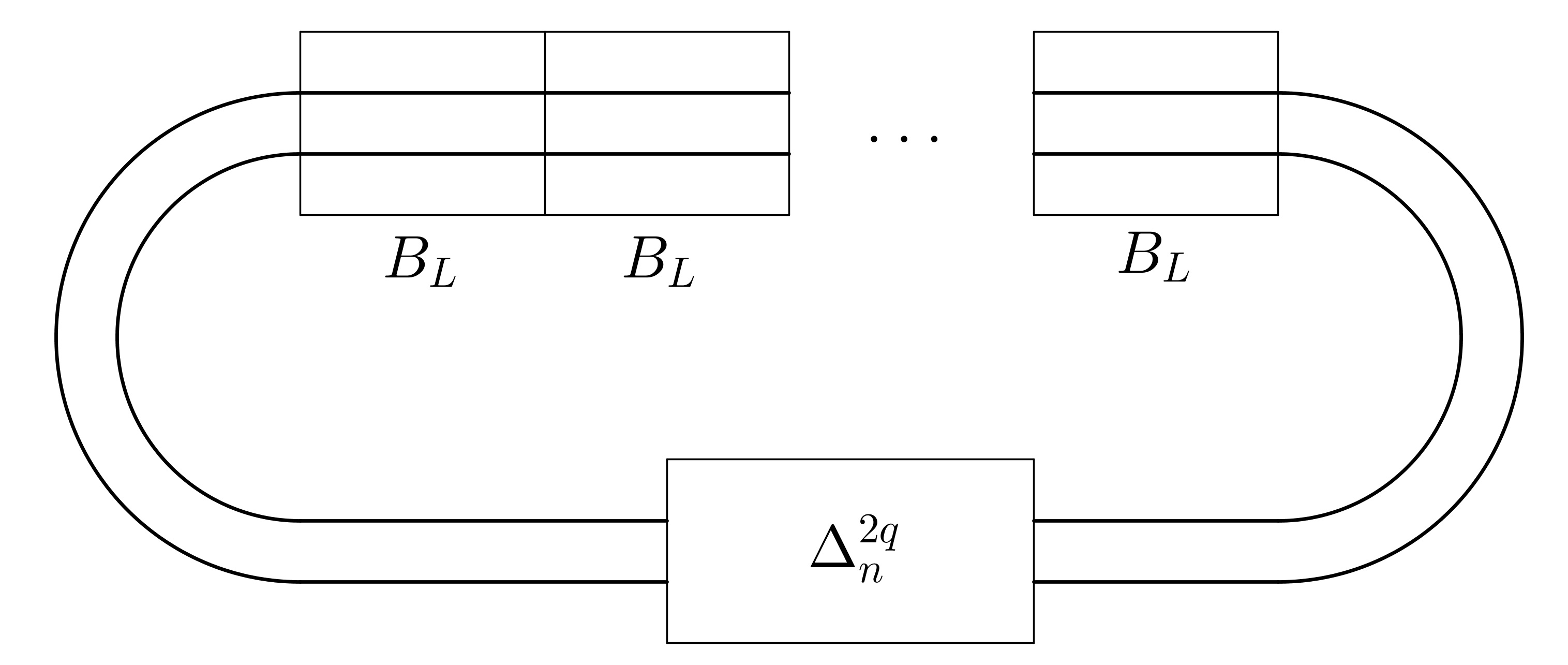}
	\caption{Diagram of the lift starting from the band diagram.}\label{linkliftprop}
\end{figure}

%% file: sezioni/sezione2.tex
\section{Fibered knots, open book decompositions and contact structures}

From now until the end of the section $M^3$ stands for a closed orientable 3-dimensional manifold.

\subsection{Fibered knots and links}

Intuitively a fibered link in $M^3$ is a link whose complement is a locally trivial bundle over $S^1$. However, this assumption it is not sufficient to give a correct definition because we have also to require that the bundle map behaves nicely \mbox{near L.}  


\begin{deft} \label{fibered}
A link $L$ in $M^3$ is a \textit{fibered link} if the two following conditions hold:
\begin{itemize}
	\item the complement of the link is the total space of a locally trivial bundle over the base space $S^1$, i.e. there exists a map $p: M^3 \setminus L \rightarrow S^1$ which is a locally trivial bundle.
	\item for each component $L_i$ of the link $L$ there exists a trivializing homeomorphism $\theta:\nu(L_i) \rightarrow S^1 \times D^2$ such that the following diagram commutes
	\[
	\xymatrix{ 
	\nu(K) \setminus K \ar[d] ^{p} \ar[r]^-{\theta} & S^1 \times (D^2 \setminus \{0\}) \ar[dl]^{\pi}\\
	S^1\\
	}
	\]
where $\pi(x,y):=\frac{y}{|y|}$.

\end{itemize}
\end{deft} 

\begin{oss}
Two fibered link which are representative of the same homeo-equivalence class may have different fibration. Two homeo-equivalent link $L_1$, $L_2$ will have the same fibration, i.e. they will be equivalent as fibered link, if the homeomorphism which realizes the equivalence respects the fibration maps.
\end{oss}

A generic bundle $E$ with fiber $F$ over $S^1$ can be described as
\[
E = \frac{F \times [0,2\pi]}{(x,0) \sim (h(x),2\pi)}
\]
where $h: F \rightarrow F$ is a homeomorphism of the fiber and $\sim$ indicates that we are identifying the corresponding points. This follows from the fact that $S^1$ is homeomorphic to $[0,2\pi]$ with endpoints identified and the interval $[0,2\pi]$ is contractible, so every bundle with fiber $F$ over $[0,2\pi]$ is trivial. The homeomorphism $h: F \rightarrow F$ is called \textit{monodromy}. Moreover, if we define the \textit{mapping torus} relative to the map $h$ as the quotient space
\[
T_h := \frac{F \times [0,2\pi]}{ (x,0) \sim (h(x),2\pi)},
\]
we can think of $E$ as the mapping torus of the associated monodromy map $h$.
\\
All the previous considerations allow us to define the monodromy associated to a fibered knot as the monodromy map relative to the knot complement.

\begin{es} \label{trivial}
Let $S^3$ be the unit sphere in $\mathbb{C}^2$, and $(z_1,z_2)=(\rho_1 e^{i\theta_1},\rho_2e^{i\theta_2})$ be two coordinate systems.
If we consider the unknot $U=\{z_1=0\}=\{\rho_1=0\}$, the complement fibers:
\[
\pi_U: S^3 \setminus U \rightarrow S^1, \hspace{5pt} (z_1,z_2) \mapsto \frac{z_1}{|z_1|}.
\]

The equivalent expression of this fibration in terms of polar coordinates is given by $\pi_U(\rho_1 e^{i\theta_1}, \rho_2 e^{i\theta_2})=\theta_1$.\\
If we represent $S^3$ as spanned by the rotation of the 2-sphere $\mathbb{R}^2 \cup \{ \infty \}$ around the circle $l \cup \{ \infty \}$, the point $P$ generates the unknot $U$ and each arc connecting $P$ and $P'$ generates a open 2-dimensional disk (see Figure \ref{funknot}). These disks exhaust the knot complement, are disjoint and parametrized by the circle $l \cup \{ \infty \}$. Therefore, $S^3 \setminus U \cong S^1 \times \text{int}(D^2)$. The fiber of the projection map $\pi_U$ results to be homeomorphic to the interior of the disk $D^2$ and the associated monodromy is the identity map on the disk.

\begin{figure}[!h]
	\centering
		\includegraphics[width=5 cm]{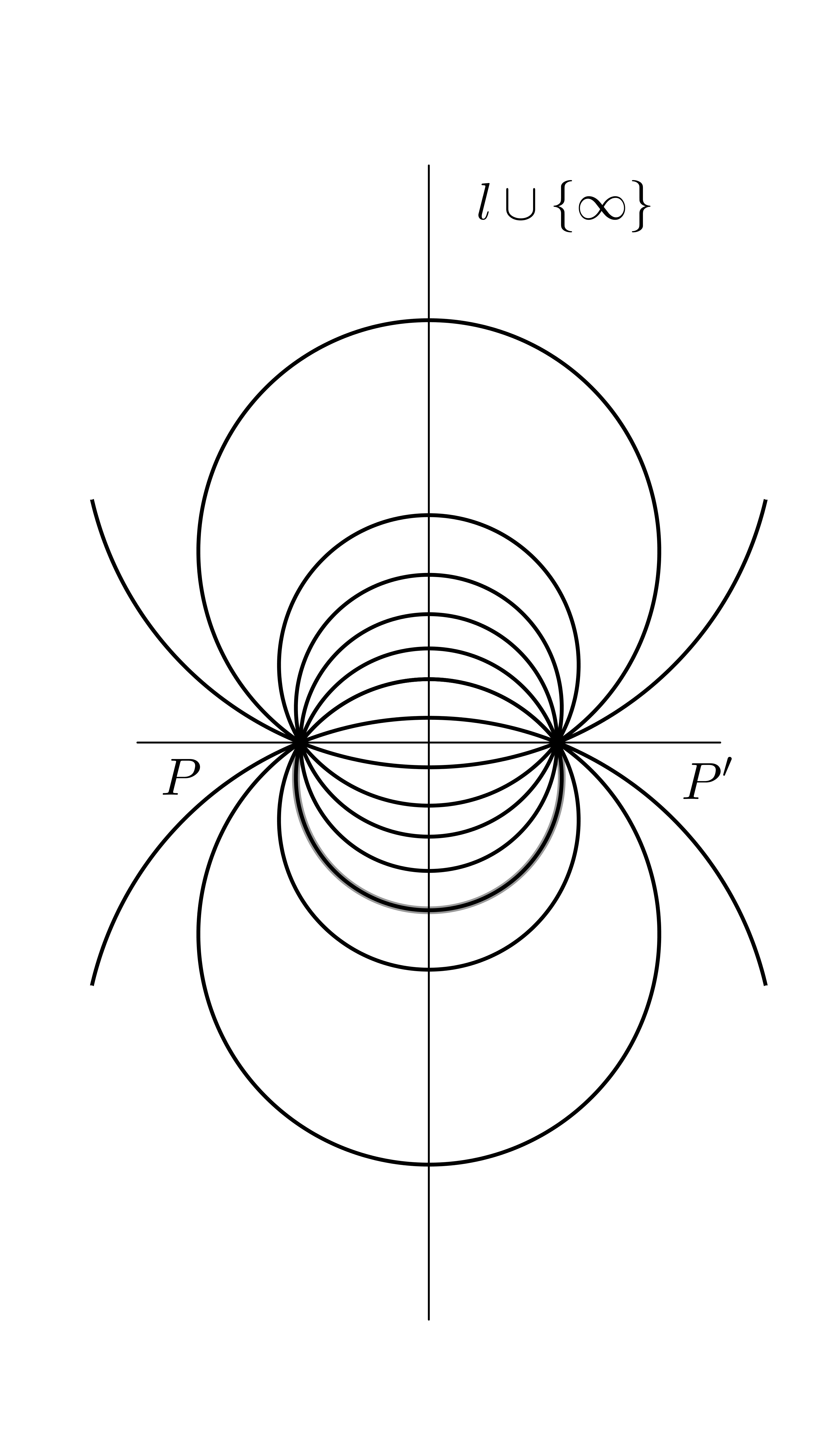}
	\caption{Fibering the unknot.}\label{funknot}
\end{figure}

\end{es}


\begin{es}\label{hband}
Using the coordinate systems adopted in the previous example, let $H^+=\{(z_1,z_2)\in S^3: z_1z_2=0\}$ and $H^-=\{(z_1,z_2)\in S^3: z_1\bar{z_2}=0\}$. We will call $H^+$ \textit{positive Hopf link} and $H^-$ \textit{negative Hopf link}, respectively. These links are represented in Figure \ref{hopfband}.

\begin{figure}[!h]
	\centering
		\includegraphics[width=6.5 cm]{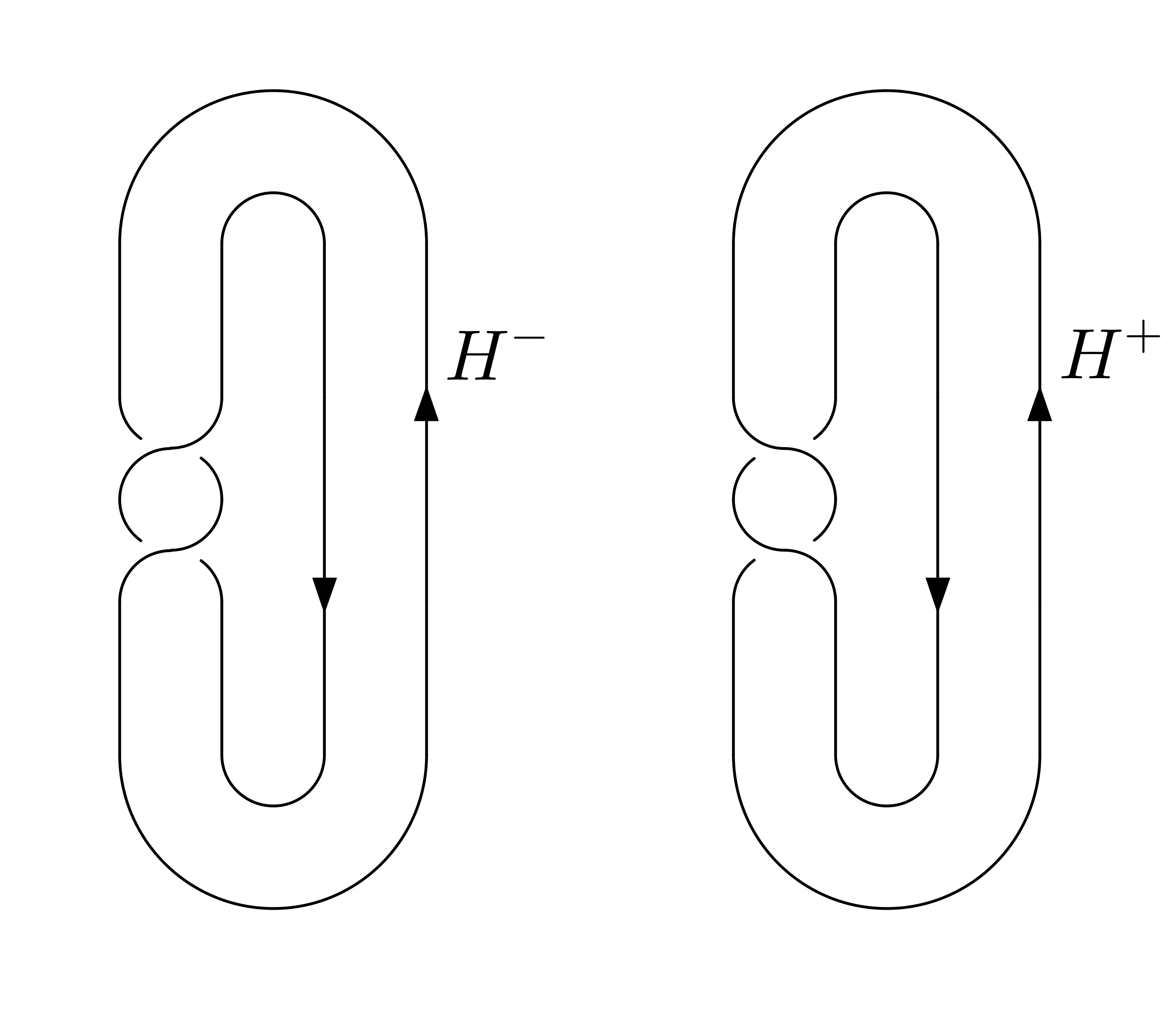}
	\caption{Diagrams of $H^+$ and $H^-$.}\label{hopfband}
\end{figure}

The following maps give the desired fibrations
\begin{align*}
&\pi^+: S^3 \setminus H^+ \rightarrow S^1, \hspace{5pt} (z_1,z_2) \mapsto \frac{z_1z_2}{|z_1z_2|}\\
&\pi^-: S^3 \setminus H^- \rightarrow S^1, \hspace{5pt} (z_1,z_2) \mapsto \frac{z_1\bar{z_2}}{|z_1\bar{z_2}|},
\end{align*}
and, when written in polar coordinates, they become $\pi ^\pm (\rho_1 e^{i\theta_1}, \rho_2 e^{i\theta_2})=\theta_1 \pm \theta_2$. The fibers of the map $\pi^\pm$ are homeomorphic to the interior of surfaces with boundary called \textit{positive Hopf band} $A^+$ and \textit{negative Hopf band} $A^-$, respectively. Both of them result to be homeomorphic to the annulus $S^1 \times [0,1]$. \\
In order to define the associated monodromy, it is useful to realize the annulus as the complex subspace $A:=\{ z \in \mathbb{C}: 1 \leq |z| \leq 2\}$. We define the \textit{positive Dehn twist along the core circle $\gamma:=\{ |z|=3/2 \}$} as the map given by
\[
D_\gamma: A \rightarrow A, \hspace{10pt} D_\gamma(\rho e^{i\theta})=\rho e^{i(\theta + 2 \pi (\rho -1))}.
\]

Since this map is actually a homeomorphism, we can invert it and its inverse is called \textit{negative Dehn twist along the core circle $\gamma$}. The image of a positive Dehn twist is reported in Figure \ref{dehntwist}. In the case of the maps $\pi^\pm$, the monodromy map coincides with a Dehn twist along the core circle of the annulus, positive or negative respectively.

\begin{figure}[!h]
	\centering
		\includegraphics[width=10.5 cm]{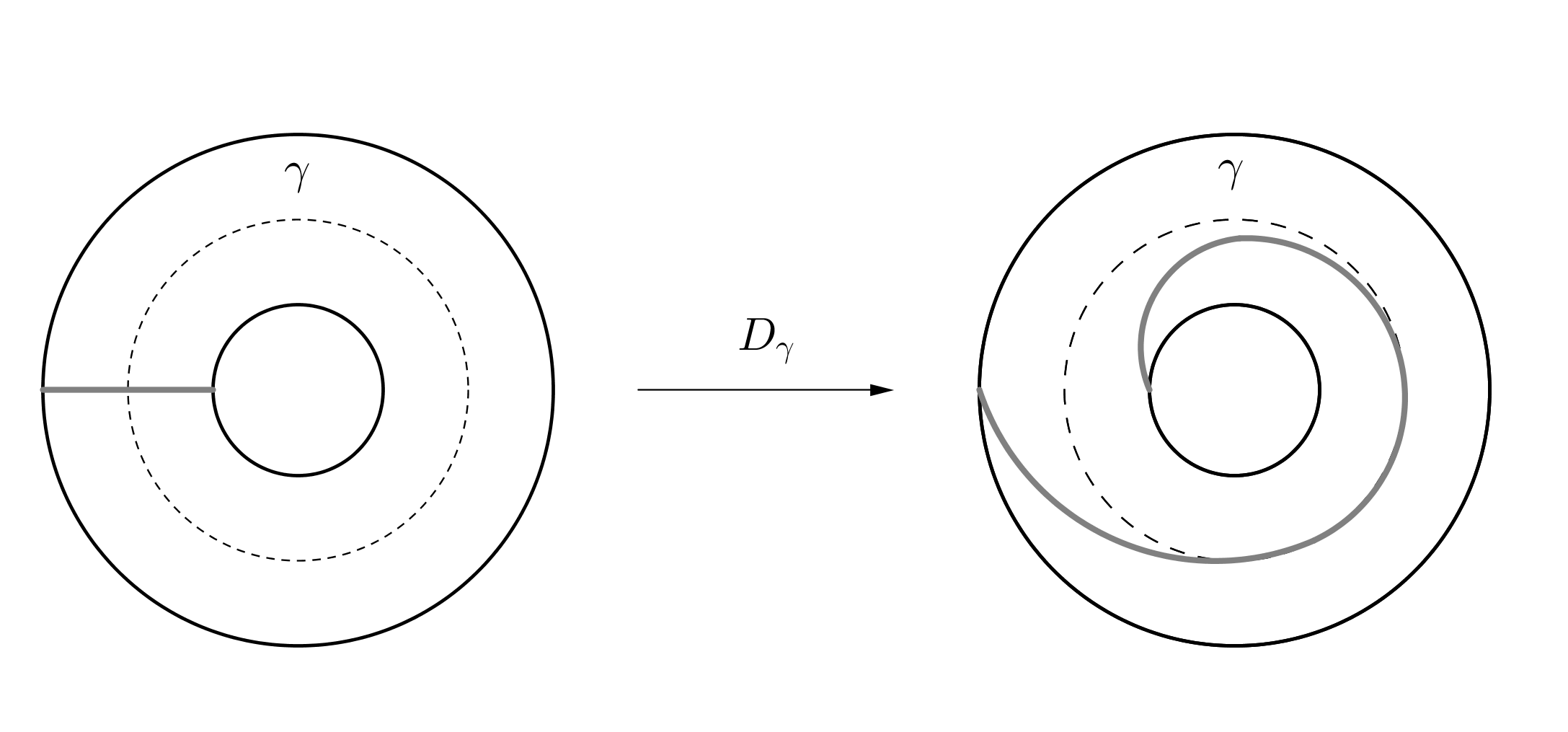}
	\caption{Positive Dehn twist about the core circle.} \label{dehntwist}
\end{figure}
\end{es}

Our aim now is to describe two moves, called plumbing and twisting, which enable us to construct all the possible fibered knots in $M^3$ starting from a fixed one. From now until the end of the section we will refer to a fibered link $L$ as a pair $(\overline{F},L)$, where $\overline{F}$ is the closure of the fiber $F$ relative to $L$ in $M^3$.\\
Let us now introduce two constructions for fibered links that change the link but maintain the property of being fibered:
\begin{description}
	\item $(A)$ \textbf{Plumbing}\\
	Suppose $(\overline{F},L)$ is a fibered pair in the particular case that the ambient \mbox{manifold} is the 3-sphere $S^3$ and let $\alpha$ be an arc included in $\overline{F}$ with endpoints on $\partial F$. We can find a 3-ball $D_\alpha \subset S^3$ such that
	\begin{enumerate}
		\item $D_\alpha \cap \overline{F}  \subset \partial D_\alpha$,
		\item there is an embedding $\varphi_\alpha: [-1,1] \times [-1,1] \rightarrow \overline{F}$ with \mbox{$\varphi_\alpha([-1,1]\times \{0\})=\alpha$} and $\text{Im}(\varphi_\alpha)= D_\alpha \cap \overline{F}$.
	\end{enumerate}
	In another copy of $S^3$ we take a Hopf pair $(A^\pm,H^\pm)$, either positive or negative, with an arc $\beta \subset A^\pm$ connecting the two boundary components. Again we find a 3-ball $D_\beta$ with the same properties as above and its corresponding map $\varphi_\beta$.

	\begin{figure}[!h]
	\centering
		\includegraphics[width=13.5 cm]{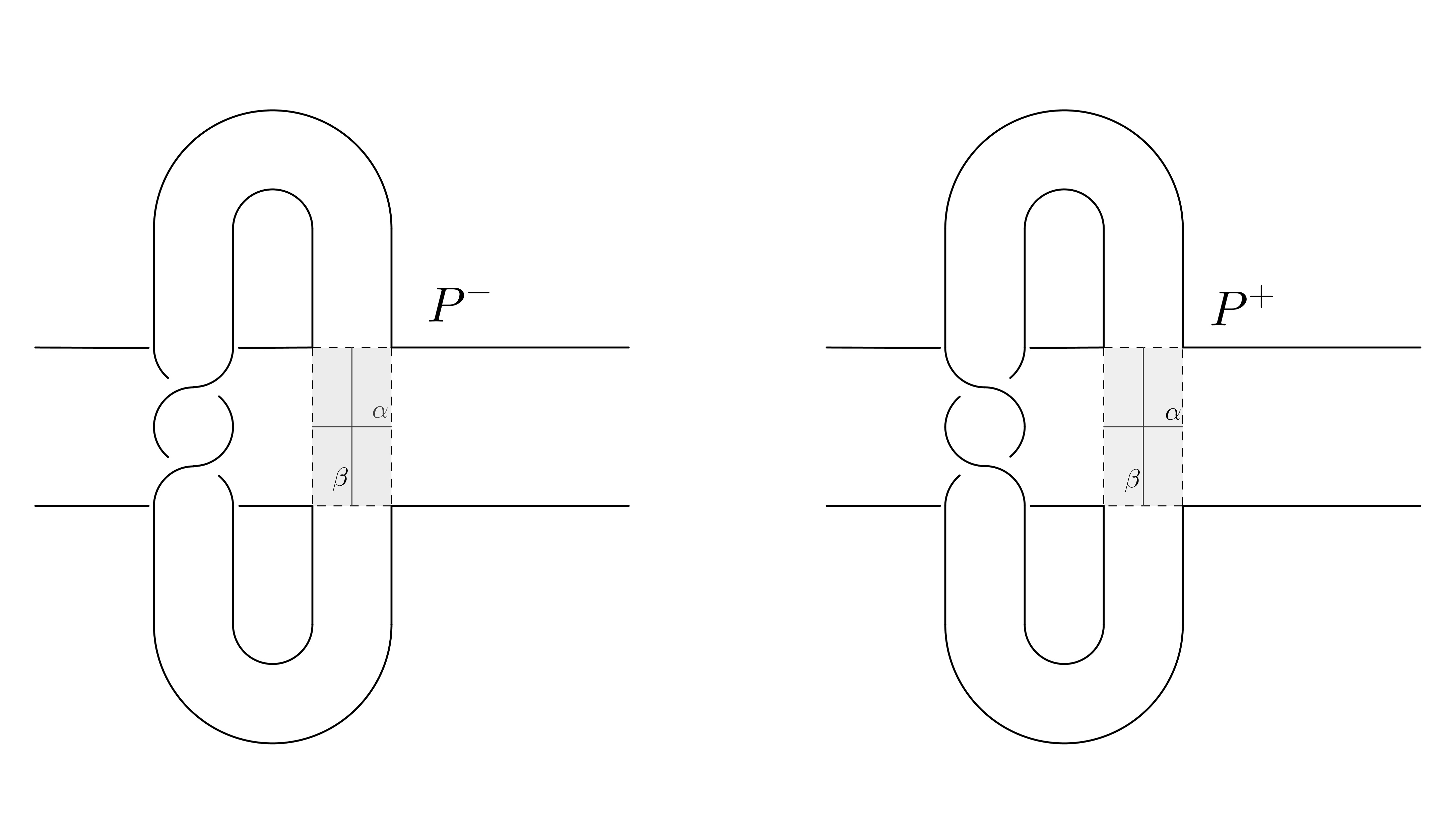}
	\caption{Plumbing.}
	\end{figure}

	Identify $\partial D_\alpha$ with $\partial D_\beta$ by an orientation reversing map $f$ satisfying 						$f(\varphi_\alpha(x,y))=\varphi_\beta(y,x)$. Then form 
	\[
	(S^3 - \text{int} (D_\alpha)) \cup_f (S^3 - \text{int}(D_\beta))
	\]
	and
	\[
	\overline{F} \cup_{f|\text{Im}(\varphi_\alpha)} A^\pm.
	\]

	The result is a new surface $\overline{F'} \subset S^3$ with boundary $L'$. In~\cite{stallings:articolo} it is proved that the pair $(\overline{F'},L')$ is fibered. The monodromy of this new pair may be chosen as $h 
	\circ D^\epsilon_\gamma$, where $h$ is the monodromy for $L$ on $\overline{F}$ and the identity on $\overline{F'} - \overline{F}$, $\gamma$ is the core circle of $A^\pm$ and $\epsilon = \pm 1$.
	Since the described operation is local, it can be easily extended to an arbitrary 3-manifold $M^3$ by restricting to a suitable coordinate chart.
	
	\item $(B)$ \textbf{Twisting}\\	
Let $(\overline{F},L)$ be a fibered pair in $M^3$ and denote by $\overline{F_1}$, $\overline{F_2}$ two parallel copies of $\overline{F}$. Let $\gamma_i \subset \overline{F_i}$ be embedded circles and put $\epsilon_i= \pm 1$. We require:
\begin{enumerate}
	\item There is an oriented annulus $A$ embedded in $M$ with $\partial A = \gamma_1 \cup \gamma_2$.
	\item Putting $\gamma_i^+$ for the oriented circle obtained by pushing off $\gamma_i$ along $\overline{F_i}$ with one extra $\epsilon_i$ full twist, $\gamma_1^+$ and $\gamma_2^+$, must intersect $A$ algebraically one point each with opposite signs.
\end{enumerate}

Surgery on $\gamma_1$ and $\gamma_2$ with framings determined by $\gamma_1^+$ and $\gamma_2^+$ will return $M^3$ and construct another fibered pair $(\overline{F'},L')$.
\end{description}
We are ready to report this result by Harer. 

\begin{teor}[{\upshape ~\cite[Theorem 2]{harer:articolo}}]
Let $(\overline{F},L)$ and $(\overline{F'},L')$ be fibered pairs in $M^3$. Then, there are pairs $(\overline{F_1},L_1)$ and $(\overline{F_2},L_2)$ such that
\begin{enumerate}
	\item $(\overline{F_2},L_2)$ is obtained from $(\overline{F},L)$ using operation $(A)$,
	\item $ (\overline{F_1},L_1)$ is constructed from $(\overline{F'},L')$ using $(A)$,
	\item $(\overline{F_1},L_1)$ may be changed into $(\overline{F_2},L_2)$ using $(B)$.
\end{enumerate}
\end{teor}

\subsection{Open book decomposition}
A fibered link $L$ in $M^3$ can be thought of as the binding of a book whose pages are the fibers of the bundle map. This intuition is formalized by the following definition.

\begin{deft}\label{obdefinition}
An \textit{open book decomposition} of $M^3$ is a pair $(B,\pi)$ where
\begin{enumerate}
	\item $B$ is an oriented link in $M^3$ called \textit{binding} of the open book,
	\item $\pi: M^3 \setminus B \rightarrow S^1$ is a locally trivial bundle such that $\pi^{-1}(\theta)$ is the interior of a surface $\Sigma_\theta$ and $\partial \Sigma_\theta=B$ for all 
	$\theta \in S^1$. The surface $\Sigma = \Sigma_\theta$ is called the \textit{page} of the open book.
\end{enumerate}
\end{deft}

Again, we will refer to the monodromy of the decomposition $(B,\pi)$ intending the monodromy of the associated bundle $\pi: M^3 \setminus B \rightarrow S^1$. The previous definition is sometimes called non-abstract open book decomposition in order to distinguish it from the notion of abstract open book decomposition.

\begin{deft}
An \textit{abstract open book decomposition} is a pair $(\Sigma,h)$ where
\begin{enumerate}
	\item $\Sigma$ is an oriented compact surface with boundary,
	\item $\phi: \Sigma \rightarrow \Sigma$ is a diffeomorphism such that $\phi$ is the identity in a neighborhood of $\partial \Sigma$.
\end{enumerate}
\end{deft}

From now until the end of the section we refer to~\cite{etnyre:articolo} for details. First of all we observe that given an abstract open book $(\Sigma,h)$ we get a 3-manifold $M^3_h$ as follows:

\[
M^3_{(\Sigma,h)}:=T_h \cup_\phi (\bigsqcup_{|\partial \Sigma|} S^1 \times D^2)
\]
where $|\partial \Sigma|$ denotes the number of boundary components of $\Sigma$ and $T_h$ is the mapping torus of $h$. Finally, $\cup_\phi$ means that the diffeomorphism $\phi$ is used to identify the boundaries of the two manifolds. For each boundary component $\gamma$ of $\Sigma$ the map $\phi: \partial (S^1 \times D^2) \rightarrow \gamma \times S^1$ is defined to be the unique diffeomorphism that takes $S^1 \times \{ p\}$ to $\gamma$ where $p \in \partial D^2$ and $\{q\} \times \partial D^2$ to $(\{q'\} \times [0,2\pi]/ \sim) \cong S^1$, where $q \in S^1$ and $q' \in \partial \Sigma$. We denote by $B_h$ the cores of the solid tori $S^1 \times D^2$ in the definition of $M_h$.\\
Two abstract open book decomposition $(\Sigma_1,\phi_1)$ and $(\Sigma_2,\phi_2)$ are said to be \mbox{\textit{equivalent}} if there is a diffeomorphism $F: \Sigma_1 \rightarrow \Sigma_2$ such that \mbox{$F \circ \phi_2 = \phi_1 \circ F$}.
\\
To indicate an open book decomposition of a manifold $M^3$ we will refer to the abstract open book decomposition $(\Sigma,h)$, where $\Sigma$ is the page and $h$ is the \mbox{monodromy}, instead of using $(L,\pi)$. The relation between the two pairs is given by the following lemma.

\begin{lem}\label{abstract}
We have the following basic facts about abstract and non-abstract open book decompositions:
\begin{enumerate}
	 \item An open book decomposition $(L,\pi)$ of $M^3$ gives an abstract open book $(\Sigma_\pi,h_\pi)$ such that $(M^3_{h_\pi},B_{h_\pi})$ is diffeomorphic to $(M^3,B)$.
	\item An abstract open book determines $M^3_h$ and an open book $(B_h,\pi_h)$ up to diffeomorphism.
	\item Equivalent open books give diffeomorphic 3-manifolds.
\end{enumerate}
\end{lem}

The basic difference between abstract and not-abstract open book decompositions is that when discussing not-abstract decompositions we can refer to the binding and to pages up to isotopy in $M^3$, whereas when discussing abstract decompositions we discuss them up to diffeomorphism.

\begin{deft}
Given two abstract open book decompositions $(\Sigma_i, h_i)$ with $i=0,1$, let $c_i$ be an arc properly embedded in $\Sigma_i$ and $R_i$ a rectangular neighborhood of $c_i$, i.e. $R_i= c_i \times [-1,1]$. The $\textit{Murasugi sum}$ of $(\Sigma_0,h_0)$ and $(\Sigma_1,h_1)$ is the abstract open book $(\Sigma_0,h_0) \ast (\Sigma_1,h_1)$ with page
\[
\Sigma_0 \ast \Sigma_1 := \Sigma_0 \cup_{R_0=R_1} \Sigma_1,
\]
where $R_0$ and $R_1$ are identified in such a way that $c_i \times \{ -1, 1\} = (\partial c_{i+1}) \times [-1,1]$, and the monodromy is $h_0 \circ h_1$.
\end{deft}

There is a precise relation between the manifolds $M^3_{(\Sigma_0,h_0)}$, $M^3_{(\Sigma_1,h_1)}$ and the one determined by the Murasugi sum $(\Sigma_0,h_0) \ast (\Sigma_1,h_1)$. This relation is reported in the following proposition.

\begin{prop}[{\upshape ~\cite{gabai:articolo}}]\label{msum}
The following diffeomorphism holds
\[
M^3_{(\Sigma_0,h_0)} \# M^3_{(\Sigma_1,h_1)} \cong M^3_{(\Sigma_0,h_0) \ast (\Sigma_1,h_1)}
\]
where $\#$ indicates the connected sum.
\end{prop}

\begin{deft}
A \textit{positive} (respectively \textit{negative}) \textit{stabilization} of an abstract open book decomposition $(\Sigma, h)$ is the open book decomposition with
\begin{enumerate}
	\item page $S_\pm\Sigma:=\Sigma \cup B$, where $B=[-1,1] \times [0,1]$ is a 1-handle, whose sides corresponding to $ \{ \pm 1 \} \times [0,1]$, are attached to the boundary of 			$\Sigma$ , and
	\item monodromy $S_\pm h:= h \circ D^{\pm}_c$, where $D_c$ is a positive (respectively negative) Dehn twist along a curve c which lies in $S\Sigma$ and intersects $\{ 0 \} \times [0,1]$ in B exactly once.
\end{enumerate}

We denote this stabilization $S_\pm(\Sigma,h)=(S_\pm\Sigma,S_\pm h)$, where $\pm$ refers to the positivity or negativity of the stabilization.
\end{deft}

The stabilization has a clear interpretation in terms of the Murasugi sum.

\begin{prop}\label{stable}
If we consider the abstract open book decomposition associated to the Example \ref{hband} and we indicate it with $(H^\pm, D_\gamma^\pm)$, we have
\[
S_\pm(\Sigma,h)=(\Sigma,h) \ast (H^\pm, D_\gamma^\pm).
\]

In particular, by Proposition \ref{msum}, it results $M^3_{S_\pm(\Sigma,h)}=M^3_{(\Sigma,h)}$.
\end{prop}

\subsection{Contact structures}

So far we have described only topological notions. Now, we are going to change the setting of our analysis introducing the differential viewpoint. We refer to~\cite{ozbagci:libro}.

\begin{deft}
A 1-form $\alpha \in \mathcal{A}^1(M^3)$ is said to be a \textit{positive contact form} if $\alpha \wedge d\alpha >  0$. The $2$-distribution $\xi \subset TM^3$ is an \textit{oriented contact structure} if locally it can be defined by a positive contact 1-form as $\xi = \ker \alpha$.
\end{deft}

\begin{oss}
There exists a more general way to define a contact 1-form on a 3-dimensional manifold $M^3$ by putting $\alpha \wedge d\alpha \neq 0$, but since we are interested only in positive contact forms, it suffices to define them. 
\end{oss}

According to a classical result of Frobenius, the distribution $\xi = \ker \alpha$ is integrable if and only if $\alpha \wedge d\alpha=0$. Since an integrable distribution has to be involutive, that is closed under Lie bracket, $\xi$ will be integrable if and only if $\alpha([X_1,X_2])=0$ whenever $X_i \in \ker \alpha, i=1,2$. Now, if we suppose $\xi$ integrable and we apply the Cartan formula to compute $d\alpha$, we obtain
\[
d\alpha(X_1,X_2)=\mathcal{L}_{X_2}(\alpha(X_1)) - \mathcal{L}_{X_1}(\alpha(X_2)) - \alpha([X_1,X_2])
\]
and, thanks to the hypothesis of integrability, we get $d\alpha(X_1,X_2)=0$ for all $X_1,X_2 \in \ker\alpha$. In particular, $d\alpha$ has to vanish on $\xi = \ker \alpha$. Therefore, the contact condition may be interpreted as a constraint which forces $\xi$ to be \textit{"maximally non-integrable"}.

\begin{es}
If on $\mathbb{R}^{3}$ we consider the euclidean coordinate system $(x,y,z)$, the standard contact form is given by $\alpha_{st} = dz + xdy$ and the associated contact structure is $\xi_{st}=\ker \alpha_{st}$.
\\
Adopting the cylindrical coordinates $(\rho,\theta,z)$ on $\mathbb{R}^3$, the form $\alpha = dz + \rho^2 d\theta$ is another example of contact 1-form. In fact, after the change into the classical $(x,y,z)$-coordinate system, we get $\alpha=dz+xdy-ydx$ and so
\[
\alpha \wedge d\alpha = 2 dx \wedge dy \wedge dz = 2 \rho d\rho \wedge d\theta \wedge dz
\]
which is clearly different from 0 all over the points where the coordinate change is defined. It is easy to verify that the contact plane of the associated distribution are generated by $\{ \frac{\partial}{\partial \rho}, \rho^2 \frac{\partial}{\partial z}-\frac{\partial}{\partial \theta} \}$. The planes relative to the points of the $z$-axis are horizontal, i.e. parallel to the $xy$-plane, and as we move out along any ray perpendicular to the $z$-axis the planes twist in a clockwise way. The twisting angle is an increasing function of $r$ which converges monotonically to $\pi/2$ as $r \rightarrow \infty$. Moreover, we observe that the distribution is invariant under translation along the $z$-axis and under rotation on the $xy$-plane.

\end{es}

\begin{es}
If we indicate with $p=(x_1,y_1,x_2,y_2)$ a point in $\mathbb{R}^4$ and we consider the map $F: \mathbb{R}^4 \rightarrow \mathbb{R}$ defined by $F(x_1,y_1,x_2,y_2):=x_1^2+y_1^2+x_2^2+y_2^2$, it is clear that $S^3=F^{-1}(1)$ and $T_pS^3=\ker dF_p= \ker (2x_1dx_1 + 2y_1dy_1 + 2x_2dx_2 + 2y_2dy_2)$. By identifying $\mathbb{R}^4$ with $\mathbb{C}^2$, we can define a complex structure on each tangent space of $\mathbb{R}^4$. More precisely, by putting
\[
J_p : T_p \mathbb{R}^4 \rightarrow T_p \mathbb{R}^4, \hspace{10pt} \text{with $p \in \mathbb{R}^4$}
\]
\[
J_p(\frac{\partial}{\partial x_i}):=\frac{\partial}{\partial y_i}, \hspace{10pt} J_p(\frac{\partial}{\partial y_i}):=-\frac{\partial}{\partial x_i} \hspace{15pt} \text{for $i=1,2$}
\]
 we immediately see $J_p^2 = -\text{Id}$. Let $\xi$ be the distribution defined by
\[
\xi_p = T_pS^3 \cap J_p(T_pS^3), \hspace{10pt} \text{with $p \in S^3$}.
\]

We claim that $\xi$ is a contact structure on $S^3$. We want to find a contact 1-form $\alpha$ such that $\xi=\ker \alpha$. If we consider the 1-form $df \circ J$ and we evaluate this form on the basis $\{ \partial/\partial x_i, \partial/\partial y_i \}$, we can recognize that $$-df_p \circ J_p = 2x_1 dy_1 - 2y_1dx_1 + 2x_2dy_2-2y_2dx_2.$$
Moreover, thanks to the equality $J_p^2=-\text{Id}$, it is possible to verify that $J_p(T_pS^3)=\ker (-df_p \circ J_p)$. By letting $\alpha:= -1/2 (df \circ J)|_{S^3}$, it is clear that $\xi = \ker \alpha$. In order to verify that $\alpha \wedge d\alpha >0$, we can pick up a point $p=(x_1,y_1,x_2,y_2)$ with $x_1 \neq 0, y_1 \neq 0, y_2 \neq 0$ and choose a basis for the tangent space $T_p S^3$ given by 
\[
\{ \frac{\partial}{\partial x_1} - \frac{x_1}{y_1}\frac{\partial}{\partial y_1},\frac{\partial}{\partial x_2} - \frac{x_2}{y_2}\frac{\partial}{\partial y_2}, \frac{\partial}{\partial x_1} - \frac{x_1}{y_2}\frac{\partial}{\partial y_2} \}.
\]

On this particular basis, it easy to see that $\alpha \wedge d\alpha > 0$. Hence we conclude that $\alpha=x_1dy_1-y_1dx_1 + x_2dy_2 - y_2dx_2|_{S^3}$ is a contact form on the 3-sphere. We define $\xi_{st}:= \ker \alpha$ as the \textit{standard contact structure} on $S^3$. This structure can be read in polar coordinates as $\xi_{st}=\ker (\rho_1^2 d\theta_1 + \rho_2^2 \theta_2)$.
\end{es}

\begin{deft}
Two contact 3-manifolds $(M^3,\xi)$ and $(N^3,\xi')$ are called \textit{contactomorphic} if there exists a diffeomorphism $F:M^3 \rightarrow N^3$ such that $F_*(\xi)=(\xi')$, that is $dF_p(\xi_m) = \xi'_{F(m)}$ for all $m \in M^3$. If $\xi = \ker \alpha$ and $\xi'= \ker \alpha'$, the previous condition is equivalent to the existence of a nowhere vanishing function $f: M^3 \rightarrow \mathbb{R}$ such that $F^*\alpha'=f\cdot \alpha$. Two contact structures $\zeta$ and $\zeta'$ on the same manifold $M^3$ are said to be \textit{isotopic} if there is a contactomorphism $h:(M^3,\zeta) \rightarrow (M^3,\zeta')$ which is isotopic to the identity.  
\end{deft}

A remarkable result about contact 1-forms is Darboux's theorem, which describes completely their local behaviour (we refer to{\upshape ~\cite[Theorem 4.4.1.13]{ozbagci:libro}}).

\begin{teor}
Given a contact 3-manifold $(M^3,\xi)$, for every $m \in M$ there is a neighborhood $U \subset M^3$ such that $(U,\xi|_U)$ is contactomorphic to $(V,\xi_{st}|_V)$ for some open set $V \subset \mathbb{R}^3$.
\end{teor}

Therefore, every contact structure $\xi$ on a 3-manifold $M^3$ is locally contactomorphic to the standard contact structure on $\mathbb{R}^3$. Thus, given a point $m \in M^3$, we can find a suitable chart $(U,\varphi)$ such that $m \in U$ and the local expression of $\xi$ in the coordinate system $\varphi(p)$ is equal to $(x,y,z)$, where $p \in U$,  is given by $\xi_{st}=dz + xdy$.

\subsection{Equivalence between definitions}

We have studied separately the three concepts of fibered links, open book decomposition and contact structures introduced so far. The next step is to understand how these notions relate to each other. In order to do this we fix a closed orientable 3-manifold $M^3$ and we define the sets:

\begin{align*}
\mathcal{F}&:=\{ \text{fibered links in $M^3$ up to positive plumbing} \} \\
\mathcal{O}&:=\{ \text{open book decompositions of $M^3$ up to positive stabilization} \} \\
\mathcal{C}&:=\{ \text{oriented contact structures on $M^3$ up to isotopy} \}.
\end{align*}

We start trying to establish a one to one correspondence between $\mathcal{O}$ and $\mathcal{F}$. It follows by definition that the binding $L$ of an open book decomposition is a fibered link for $M$. Viceversa, given a fibered link in $M^3$, we have automatically an open book decomposition of $M^3$. Moreover, the plumbing operation on fibered links becomes the stabilization procedure in terms of open book decompositions. Thus, it remains defined a bijective map

\[
\Pi : \mathcal{O} \rightarrow \mathcal{F}, \hspace{10pt} \Pi(\mathfrak{ob}):=B_{\mathfrak{ob}}
\]
where $B_{\mathfrak{ob}}$ is the binding of the open book decomposition $\mathfrak{ob}$.\\
The bijection between $\mathcal{O}$ and $\mathcal{C}$ reveals to be less immediate than the previous one. In order to find the desired correspondence, we need to understand when an open book decomposition and a contact structure result compatible.


\begin{deft}

A contact structure $\xi$ on $M^3$ is \textit{supported} by an open book decomposition $(B,\pi)$ of $M$ if $\xi$ can be isotoped through contact structures so that there is a contact 1-form $\alpha$ for $\xi$ such that
\begin{enumerate}

	\item $d\alpha$ is a positive volume form on each page $\Sigma$ of the open book and
	\item $\alpha >0$ when restricted to the binding $B$.

\end{enumerate}

Equivalently, we will say that the contact structure $\xi$ and the open book $(B,\pi)$ are \textit{compatible}.
\end{deft}

\begin{teor}[{\upshape ~\cite{thurston:articolo}}]\label{contex}

Every open book decomposition of a closed and oriented 3-manifold $M^3$ admits a compatible contact structure.

\end{teor}

\begin{prop}[{\upshape ~\cite[Proposition 2]{giroux:articolo}}]\label{uisotopy}

Any two contact structures compatible with a given open book decomposition are isotopic.

\end{prop}
For a proof of both see ~\cite{ozbagci:libro}. These two statements suggest that we have a well defined map
\[
\Psi : \mathcal{O} \rightarrow \mathcal{C}, \hspace{10pt} \Psi(\mathfrak{ob}):=\xi_{\mathfrak{ob}}
\]
where $\mathfrak{ob}$ is an open book decomposition of $M^3$ and $\xi_\mathfrak{ob}$ is the associated contact structure of Theorem \ref{contex}. Giroux's Theorem states that the map $\Psi$ is invertible. More precisely, we have the following

\begin{teor}[{\upshape ~\cite[Theorem 4]{giroux:articolo}}]\label{giroux}

Two isotopic contact structures are supported by two open book decompositions which have a common positive stabilization. Equivalently it holds:

\begin{description} 
	\item A)  For a given open book decomposition of $M^3$ there exists a compatible contact structure $\xi$ on $M^3$. Moreover, contact structures compatible with the same open book decomposition are isotopic.
	\item B)  For a given contact structure $\xi$ on $M^3$ there is a compatible open book decomposition of $M^3$. Moreover two open book decompositions compatible with a fixed contact structure admit a common positive stabilization.
\end{description}

\end{teor}

To sum up, thanks to Theorem \ref{giroux}, we get the following sequence of bijections
\[
\mathcal{C} \leftrightarrow \mathcal{O} \leftrightarrow \mathcal{F}
\]
realized by the maps $\Pi$, $\Psi$ and their inverses. In particular, if we are interested in studying the open book decompositions of a 3-dimensional manifold, and so the fibered links, we can equivalently study the contact structures on the same manifold and try to classify the associated open book decompositions by the use of contact properties.

%% file: sezioni/sezione3.tex
\section{Fibered knots in lens spaces}

The main goal of this section is to construct a mixed link diagram of some fibered links in lens spaces which are already known in literature in form of open book decomposition. The following lemma, which describes the effects of surgery on a fixed open book decomposition of a 3-manifold, will reveal to be essential for our purposes.

\begin{lem} [{\upshape ~\cite[Lemma 2.3.11]{onaran:tesi}}] \label{obdeffects}
Let $(\Sigma,h)$ be an open book decomposition for a closed orientable 3-manifold $M^3$.
\begin{description}
	\item 1) If $K$ is an unknotted circle in $M^3$ intersecting each page $\Sigma$ transversely once, then the result of a 0-surgery along $K$ is a new 3-manifold $M^3_0$ with an open book 	         decomposition having page $\Sigma' = \Sigma - \{ \text{open disk} \}$ (the disk is simply a neighborhood of the point of intersection of $\Sigma$ with $K$) and 
	         monodromy $h' = h|_{\Sigma'}$.
	\item 2) If $K$ is an unknotted circle sitting on the page $\Sigma$, then a $\pm 1$-surgery along $K$ gives back a new 3-manifold $M^3_1$ with open book decomposition having page
		$\Sigma$ with monodromy $h'=h \circ D_K^{\mp}$, where $D_K$ is the Dehn twist along $K$.
\end{description}

\end{lem}

\begin{proof}
We will give a sketch of the proof only for $1)$ by following Claim 2.1 of~\cite{pavelescu:articolo}.\\
Let us consider the standard open book decomposition of $S^3$ given by $(D^2,\text{Id})$ that has binding $\partial D^2 = U$ and call $U'$ the 0-framed knot. Since $U'$ has to be transverse to each page in order to satisfy the hypothesis, we can think of it as a line orthogonal to the disk $D^2$ to which we add a point at infinity (see Figure~\ref{0surgery1}$(1)$).

\begin{figure}[!h]  
	\centering
		\includegraphics[width=13cm]{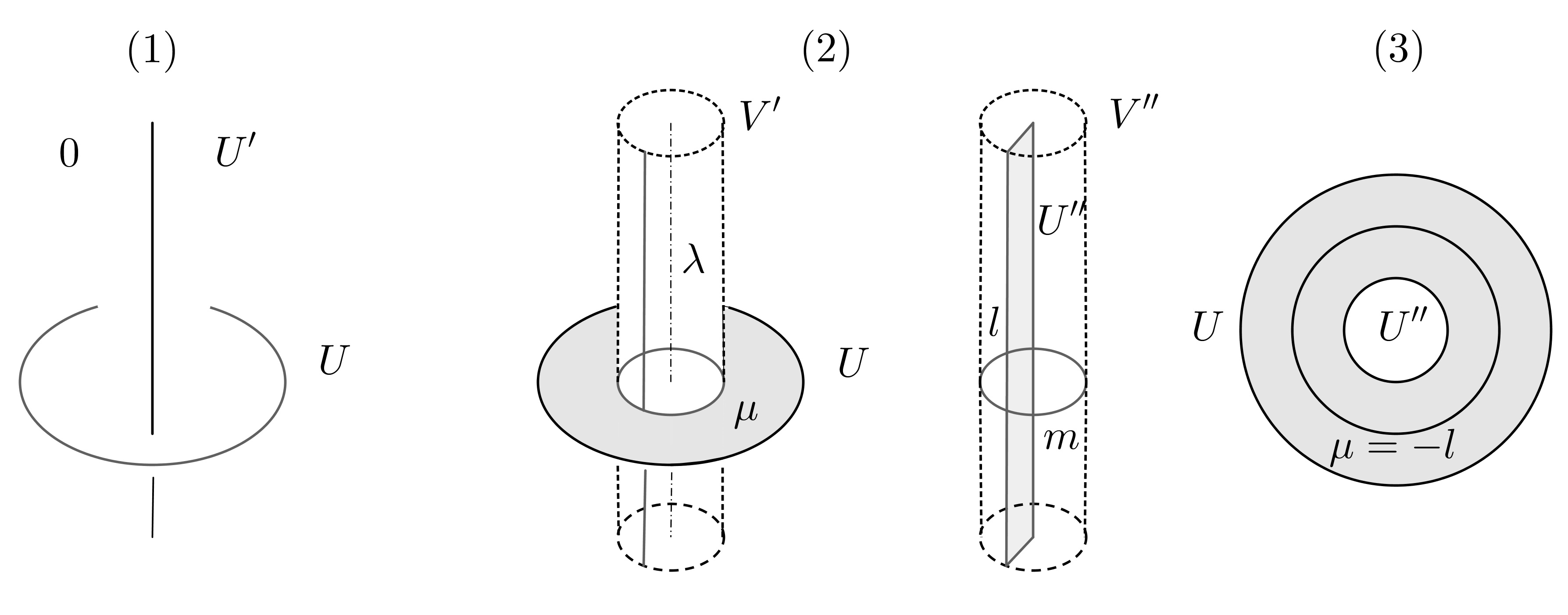}
	\caption{0-surgery on $D^2$ in $S^3$.}  \label{0surgery1}
\end{figure}

To perform a 0-surgery along $U'$ we take a thickened neighborhood $V'$ of $U'$ on which lie the meridian $\mu$ and the longitude $\lambda$ relative to $U'$, we remove $V'$ and we glue it back by sending $\mu \mapsto \lambda$. This procedure is completely equivalent to taking a new solid torus $V''$ with meridian $m$ and longitude $l$ and gluing it in $S^3$ instead of $V'$ using a homeomorphism of the boundaries which sends $m\mapsto \lambda$ and $l \mapsto -\mu$.
The open book decomposition of the resulting manifold, $S^1 \times S^2$, will have the annulus bounded by $U$ and $U''$ as page (see Figure~\ref{0surgery1}$(3)$) and the identity map as monodromy. Thus $1)$ is proved in the particular case of a 0-surgery transverse to $(D^2,\text{Id})$ in $S^3$.\\
Since the previous procedure can be easily extended to the more general case with page $\Sigma$ for a 3-manifold $M^3$, we have concluded.
\end{proof}

We will try to visualize how the open book decomposition $(D^2, \text{Id})$ of $S^3$ is modified by the 0-surgery along $U'$. Let us recall the model presented in Example \ref{trivial} thanks to which we think of $S^3$ as generated by the rotation of a 2-sphere around the circle $l \cup \{\infty\}$. Suppose to identify $U'$ with $l \cup \{ \infty \}$.\\
By removing a neighborhood $V'$ of $U'$, we puncture each page of the decomposition (see Figure~\ref{0surgery2}$(2)$). The surface obtained from the page is an annulus with two boundary components, the one generated by the point $P$ and the other one generated by the point $R$ for example, as we can see in Figure~\ref{0surgery2}. \\
By performing the 0-surgery, we attach another annulus along the boundary of each puncture. One of these annuli is generated by the arc $RQ$, represented in Figure~\ref{0surgery2}$(3)$. 
\begin{figure}[!h]
	\centering
		\includegraphics[width=13cm]{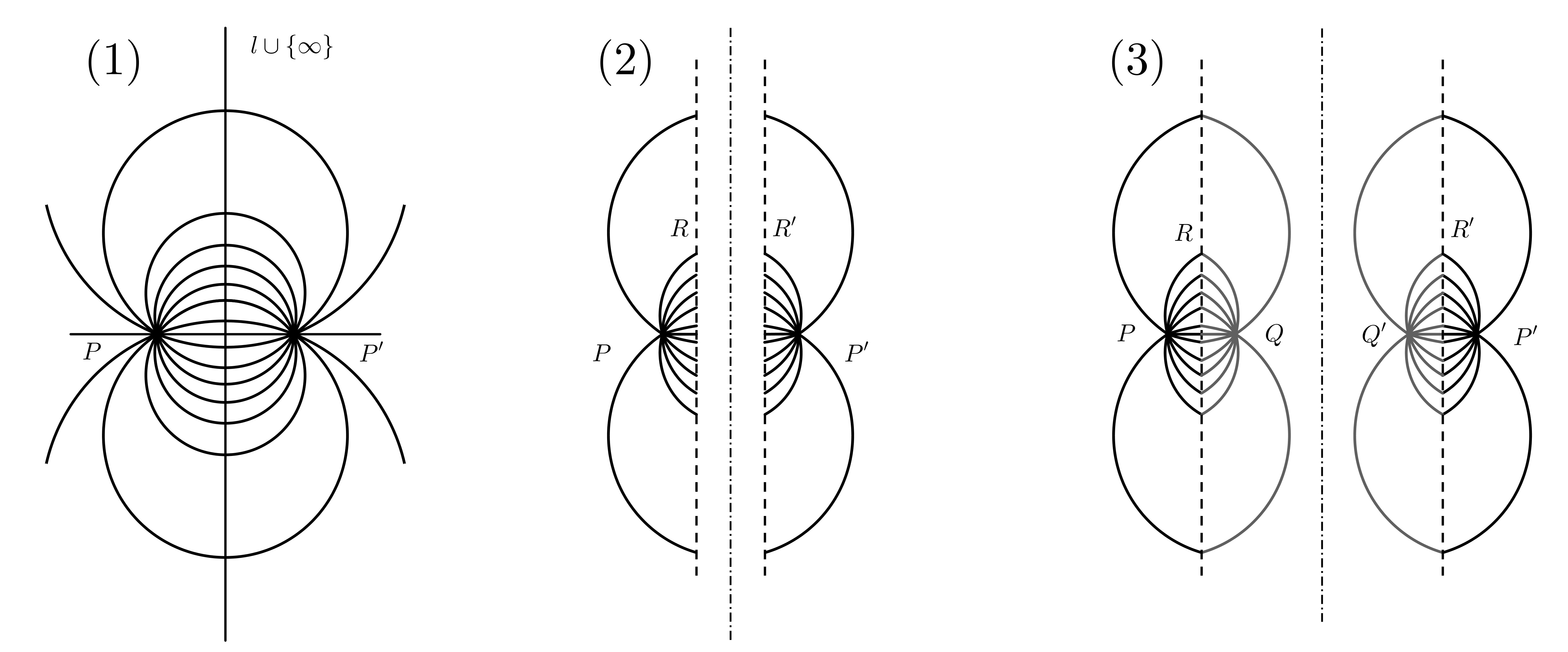}
	\caption{Construction of an open book decomposition for $S^1 \times S^2$.}  \label{0surgery2}
\end{figure}

From the previous reasoning we deduce that $S^1 \times  S^2$ is represented by a rotating 2-sphere. This sphere is thought of as a disk whose boundary points are identified by a reflection along a fixed axis. The identification is represented in Figure~\ref{quotienttorus}$(1)$.  In the same way the page of the decomposition is spanned by the rotation of the segment $PQ$ around $l \cup \{ \infty \}$ (see Figure~\ref{0surgery2}$(3)$).\\
Finally, we report an alternative representation of the decomposition of $S^1 \times S^2$ given by a mixed link diagram in Figure~\ref{quotienttorus}$(2)$.
\begin{figure}[!h]
	\centering
		\includegraphics[width=13cm]{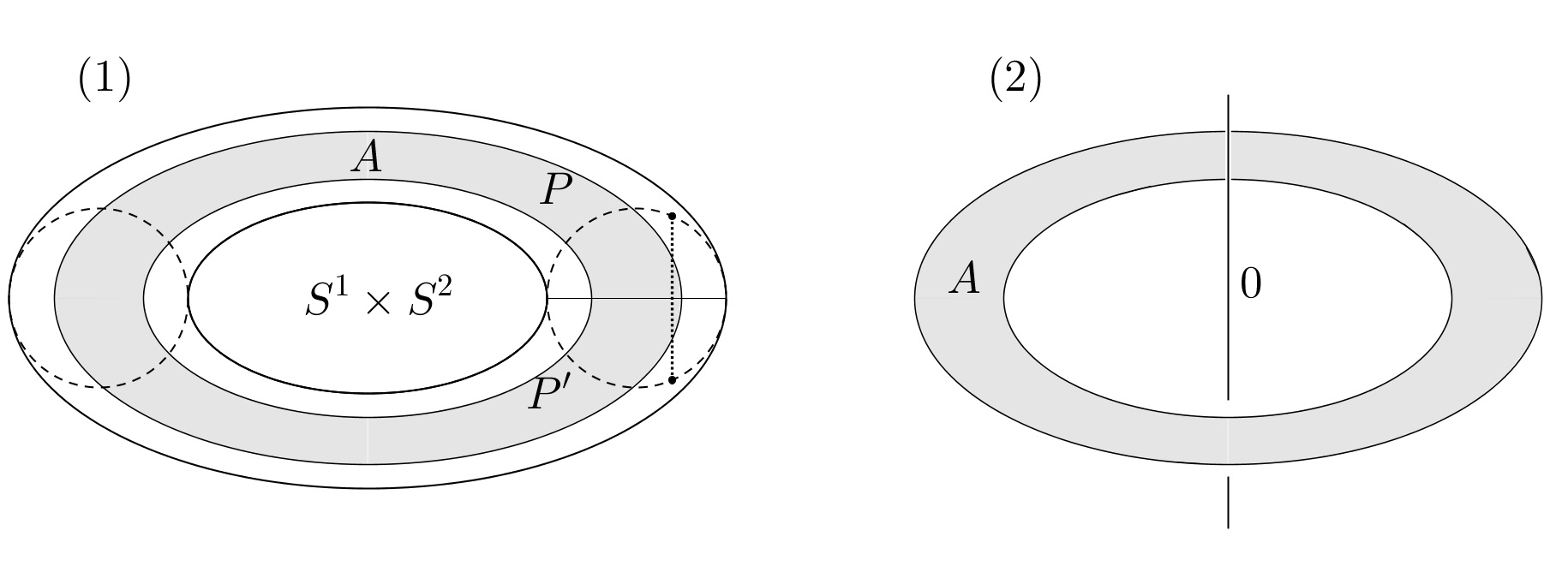}
	\caption{Representations of the open book decomposition for two different model of $S^1 \times S^2$.} \label{quotienttorus}
\end{figure}

In particular part $1)$ of Lemma \ref{obdeffects} tells us even more: it states not only that the page of the open book decomposition of $M^3$ will be punctured in correspondence of each intersection point with the surgery link, but it suggests also how the page of the new manifold $M^3_0$ lies with respect to the surgery link, i.e. it gives back a representation of the binding $L$ (boundary of the annulus $A$ in Figure~\ref{quotienttorus}$(2)$) as the moving part of a mixed link whose fixed part is the 0-framed knot $U'$.

%

Now, thanks to the notion of open book decomposition and with the help of Lemma \ref{obdeffects}, we are ready to show the key example of this survey.

\begin{prop} \label{casep1}
Given $p \in \mathbb{Z}$, there exists a fibered link $L \subset L(p,1)$ represented by the mixed link diagram represented in Figure \ref{fiberedp}. Moreover the monodromy of the link $L$ is given by $D^{-p}_\gamma$, where $\gamma$ is the core circle of the annulus $A$ bounded \mbox{by $L$.}

\begin{figure}[!h]\label{fiberedp}
	\centering
		\includegraphics[width=5cm]{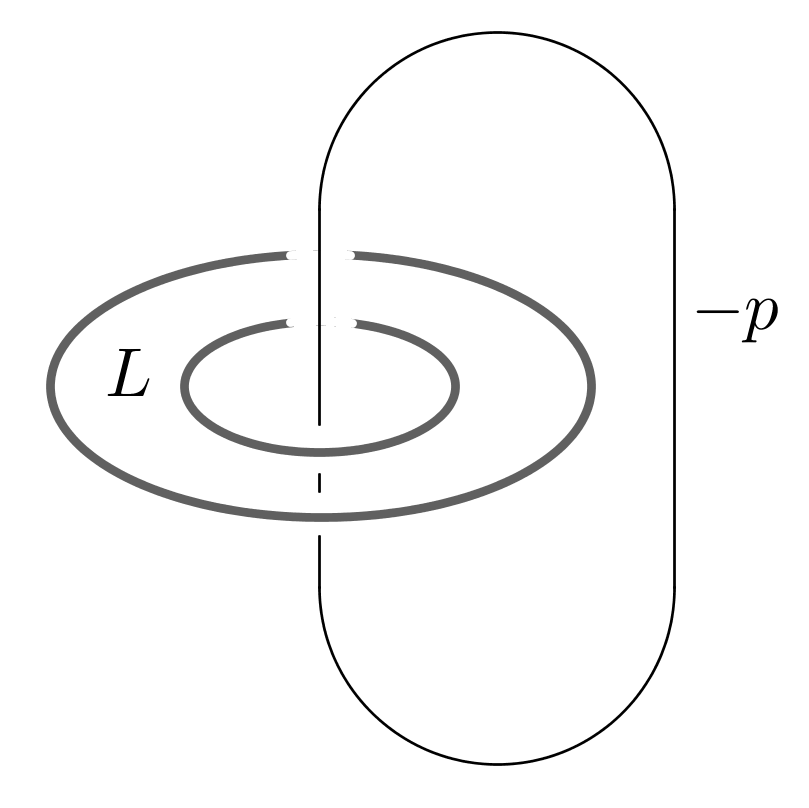}
	\caption{Fibered link in $L(p,1)$.}\label{fiberedp}
\end{figure}

\end{prop}

\begin{proof}
Let us consider the standard open book $(D^2,\text{Id})$ in $S^3$ with binding \mbox{$\partial D^2 = U$}. As before, the surgery knot links $U$ exactly once and intersect each page transversely. By adding $\pm 1$-framed components with the help of Kirby moves, we reduce the framing of the surgery knot to be 0. Thanks to Lemma \ref{obdeffects}, we know precisely how to get a mixed link representation from a 0-surgery transverse to the each page. To conclude, we can remove from the representation the $\pm 1$-components, getting back a $-p$-framed knot as in Figure~\ref{fiberedp}.\\
The computation of monodromy needs to distinguish three different cases. If $p=0$, a direct application of Lemma~\ref{obdeffects} shows that the monodromy is the restriction of the initial monodromy, which is exactly the identity in our case. Let us fix $p>0$. Each $+1$-framed circle added to reduce the framing of surgery knot to $0$ contributes to the monodromy with a $D^-_\gamma$. Since we have exactly $p$ components framed by $+1$, the monodromy results to be equal to $D^{-p}_\gamma$, as required. Finally, since the case $p<0$ differs from the previous one only by sign, we have done. 
\end{proof}

\begin{figure}[!h]
	\centering
		\includegraphics[width=13cm]{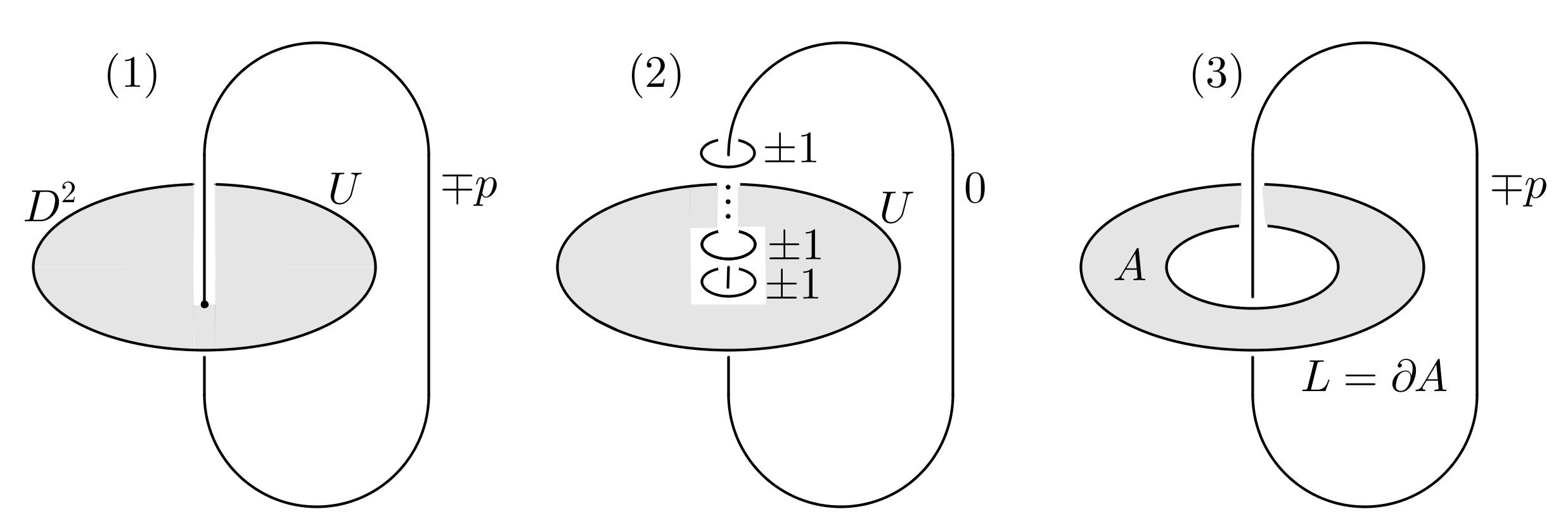}
	\caption{Construction of the fibered link $L$.}
\end{figure}

To extend our previous results to the more general case of the lens space $L(p,q)$ with $q>1$, we have to recall the presentation of lens spaces $L(p,q)$ given by integral surgery on the framed link of Figure~\ref{integralpq}. Then we modify the chosen link in order to present it as the closure of a suitable framed braid.

\begin{prop}\label{fiberedlinkprop}
If $-p/q$ has a continued fraction expansion given by $[a_1,a_2,\ldots,a_n]$ and $L(p,q)$ is presented by the framed link $L$ in Figure \ref{integralpq}, then there exists a fibered link in $L(p,q)$ given by Figure.

\begin{figure}[!h]\label{fiberedlink}
	\centering
		\includegraphics[width=14cm]{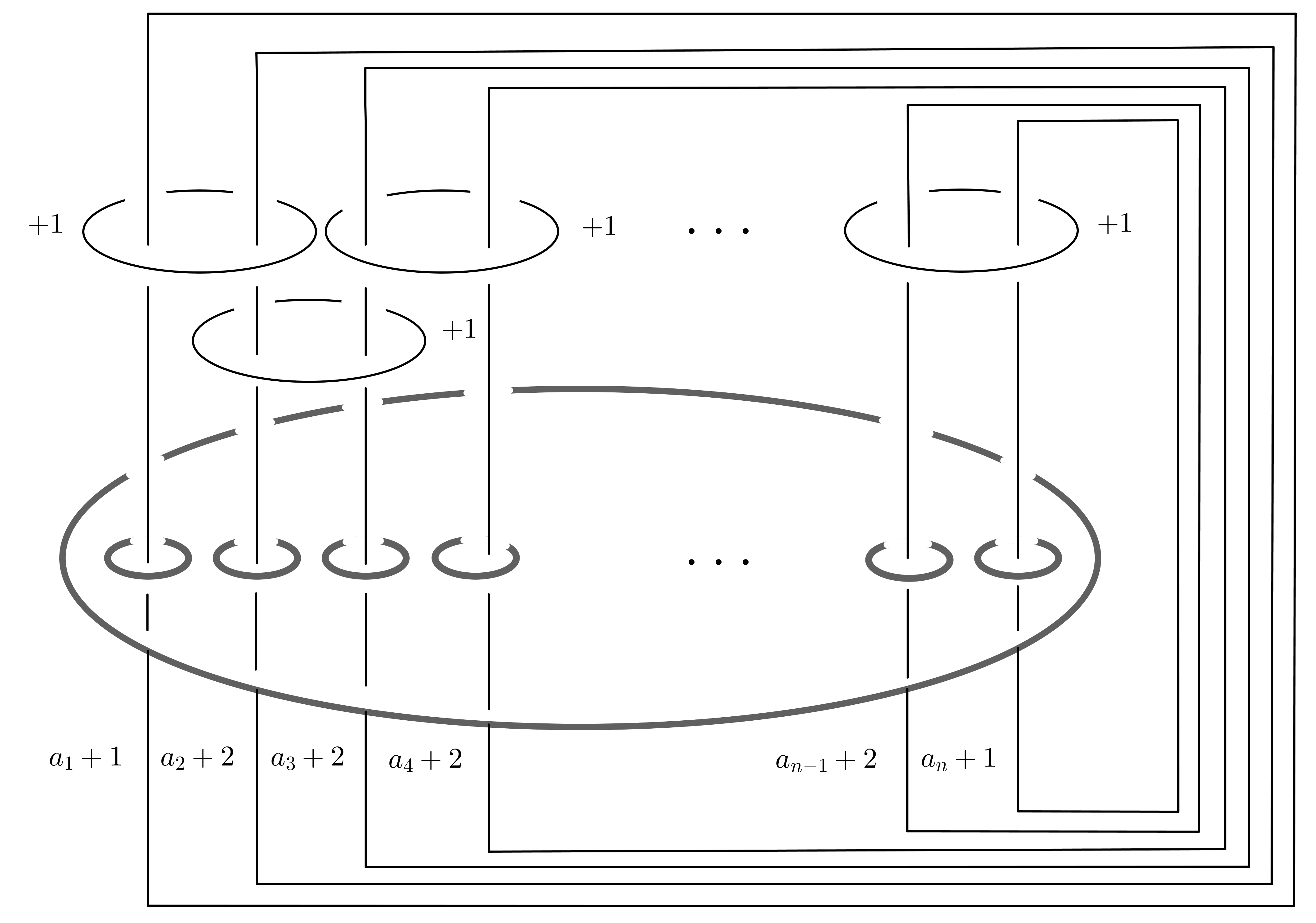}
	\caption{Fibered link in $L(p,q)$.}\label{fiberedlink}
\end{figure}
\end{prop}

\begin{proof}
The proof is a clear extension of the case $q=1$. Let us consider the standard open book decomposition $(D^2,\text{Id})$ of $S^3$ with binding $U$. We require that $U$ links each component of the framed link exactly once and every component has to intersect the pages transversely.\\
Firstly, we remove each linking of $L$ by introducing $+1$ framed circles. Then, we keep adding $\pm 1$-framed circles to reduce to $0$ each surgery framing. By applying Lemma \ref{obdeffects}, we get a mixed link from each $0$-framed transverse surgery. Finally, we remove every $\pm 1$-component, to get the desired presentation.
\end{proof}

\begin{es}
We will give an example for the case $p/q=[a_1,a_2,a_3]$. Let us suppose to have a lens space $L(p,q)$ represented by a framed link with three components. The first step of the algorithm is to represent the framed link as a closed braid (see Figure~\ref{case3}$(1)$). Then, we consider the standard open book decomposition $(D^2,Id)$ of $S^3$ with $D^2$ in the position required by Proposition~\ref{fiberedlinkprop} with respect to the framed link. In order to remove each linking between the components we have to introduce $+1$-framed circles. To change $a_1 +1, a_2 +1$ and $a_3 +1$ into $0$, we keep introducing $\pm1$-framed circles (see Figure~\ref{case3}$(2)$). At this point, we apply the first part of Lemma~\ref{obdeffects} to the page $D^2$ and we obtain a page punctured three times. Finally, we remove each circle previously added in order to get back the framed link shown in Figure~\ref{case3}(3).

\begin{figure}[!h]
	\centering
		\includegraphics[width=14cm]{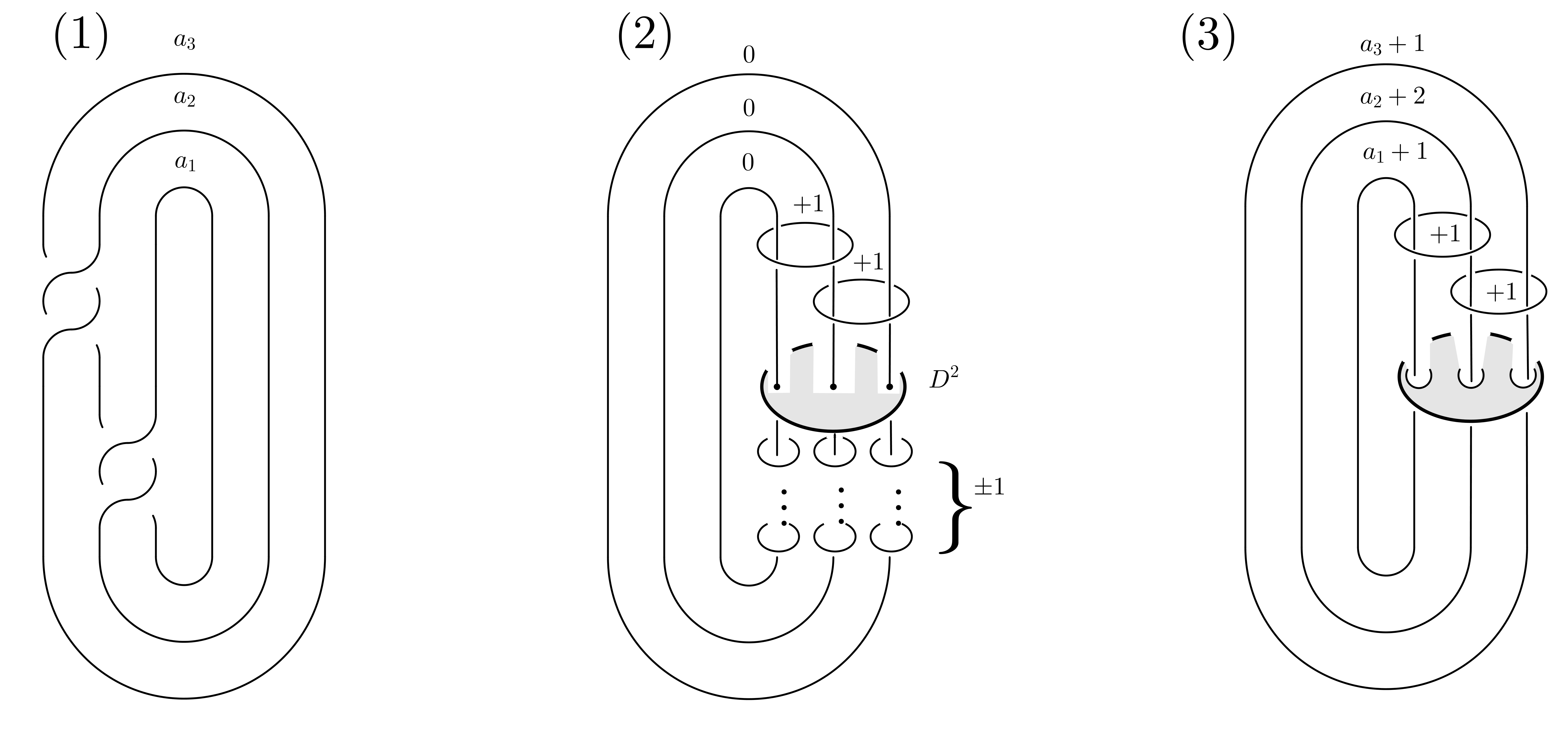}
	\caption{Construction of a fibered link for $p/q=[a_1,a_2,a_3]$}\label{case3}
\end{figure}
\end{es}

\begin{oss}
In the particular case described in Proposition \ref{casep1} we can apply a plumbing operation to $L$ to get the a fibered knot in $L(p,1)$. We represent it by the punctured disk diagram in Figure \ref{fiberedtoric}.

\begin{figure}[!h]\label{fiberedtoric}
	\centering
		\includegraphics[width=4.5cm]{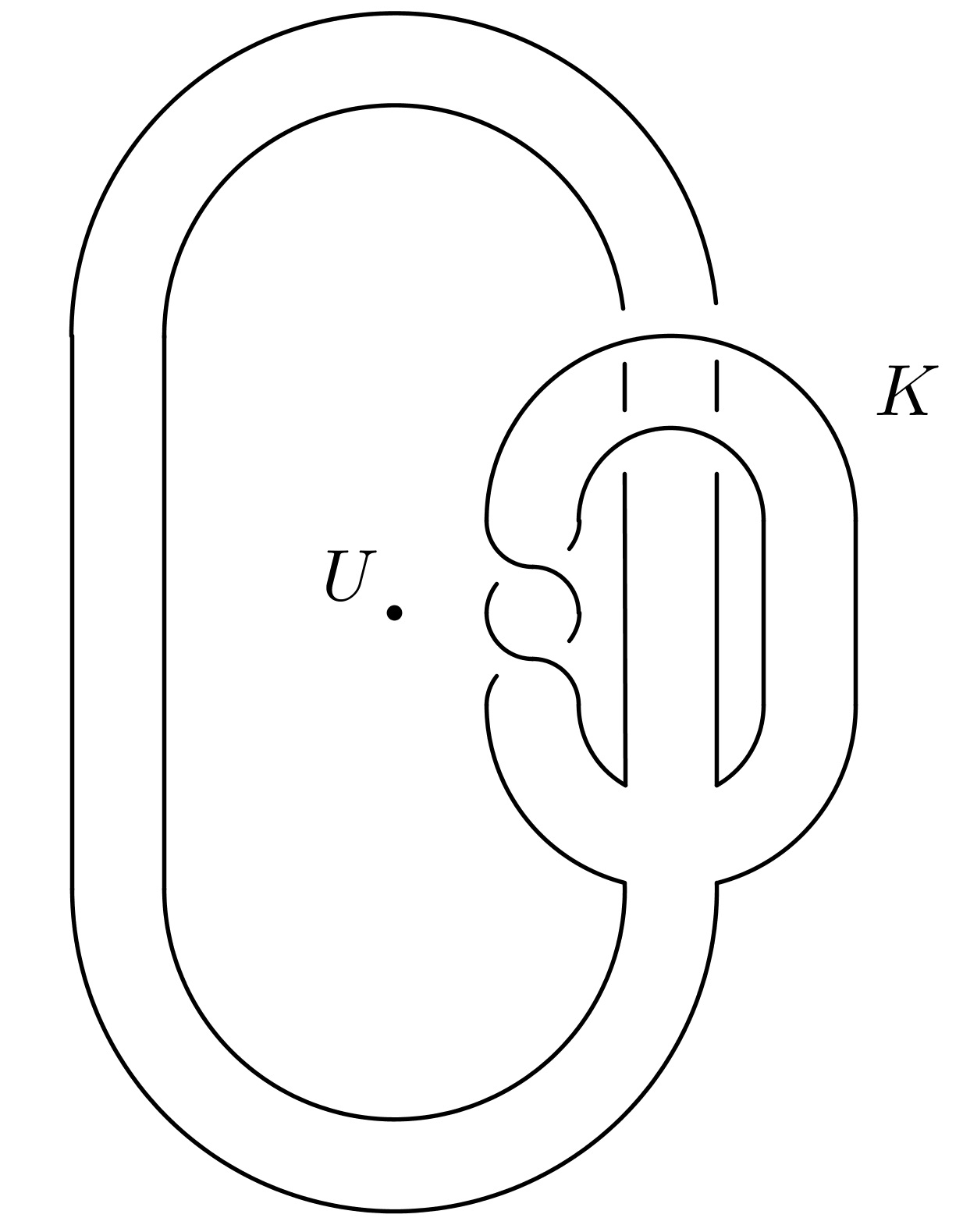}
	\caption{Punctured disk diagram for the fibered knot $K$.}\label{fiberedtoric}
\end{figure}

An analogous procedure involves the more general case of $L(p,q)$. If we consider the example reported in Proposition \ref{fiberedlinkprop}, we need to repeat the plumbing operation several times, once for each integer $a_i$ that appears in the continued fraction decomposition.
\end{oss}

%% file: sezioni/sezione4.tex
\section{Lift in the 3-sphere of fibered links in $L(p, 1)$}

In the first part of this section we are going to exhibit an example of the lift of the link exposed in Proposition~\ref{casep1} for both $p<1$ and $p>1$. The lift will be a negative Hopf link $H^-$ for $p>1$ and a positive Hopf link $H^+$ for $p<1$. Since in both cases the lift reveals to be a fibered link, we can ask if the links of Proposition~\ref{casep1} may be obtained by quotienting the Hopf pairs in $S^3$ by the $\mathbb{Z}_p$-action.\\
We dedicate the second part of the section to the explicit computation which gives a positive answer to the previous question. Since we are also interested in translating topological results in terms of contact structures, we will show that the standard contact structure on $S^3$ is $\mathbb{Z}_p$-invariant and determines a standard contact structure on all the lens spaces $L(p,q)$. Moreover, this new contact structures reveal to support the open book decomposition associated to the quotient of $H^+$ when $p<1$.\\
A priori, it seems we have reached two different open book decompositions for $L(p,-1)$: the one associated to the link reported in Proposition~\ref{casep1} and the one obtained by quotienting $H^+$ under the $\mathbb{Z}_p$-action. We conclude the section by showing that these decompositions are equivalent.\\
In order to determine a lift for the link in Proposition~\ref{casep1}, we denote that link by $L$. To get the correct lift it is important to orient $L$ thinking of it as the boundary of a generic page of the fibration.
\\
Let us consider the case $p>1$. First of all we want to underline that ~\cite{manfredi:tesi} represents the lens space $L(p,q)$ as obtained by a rational surgery along the unknot with coefficient $p/q$, and not $-p/q$, as we have supposed so far. Thus, in order to apply correctly Proposition~\ref{lift}, we have to modify the framing $-p$. By following Section 16.4 of ~\cite{prasolov:libro}, we immediately understand that we can change the coefficient $-p$ into $p/p-1$ but we must add two extra twists to the link $L$ (see Figure~\ref{rolfsen}).

\begin{figure}[!h]
	\centering
		\includegraphics[width=10cm]{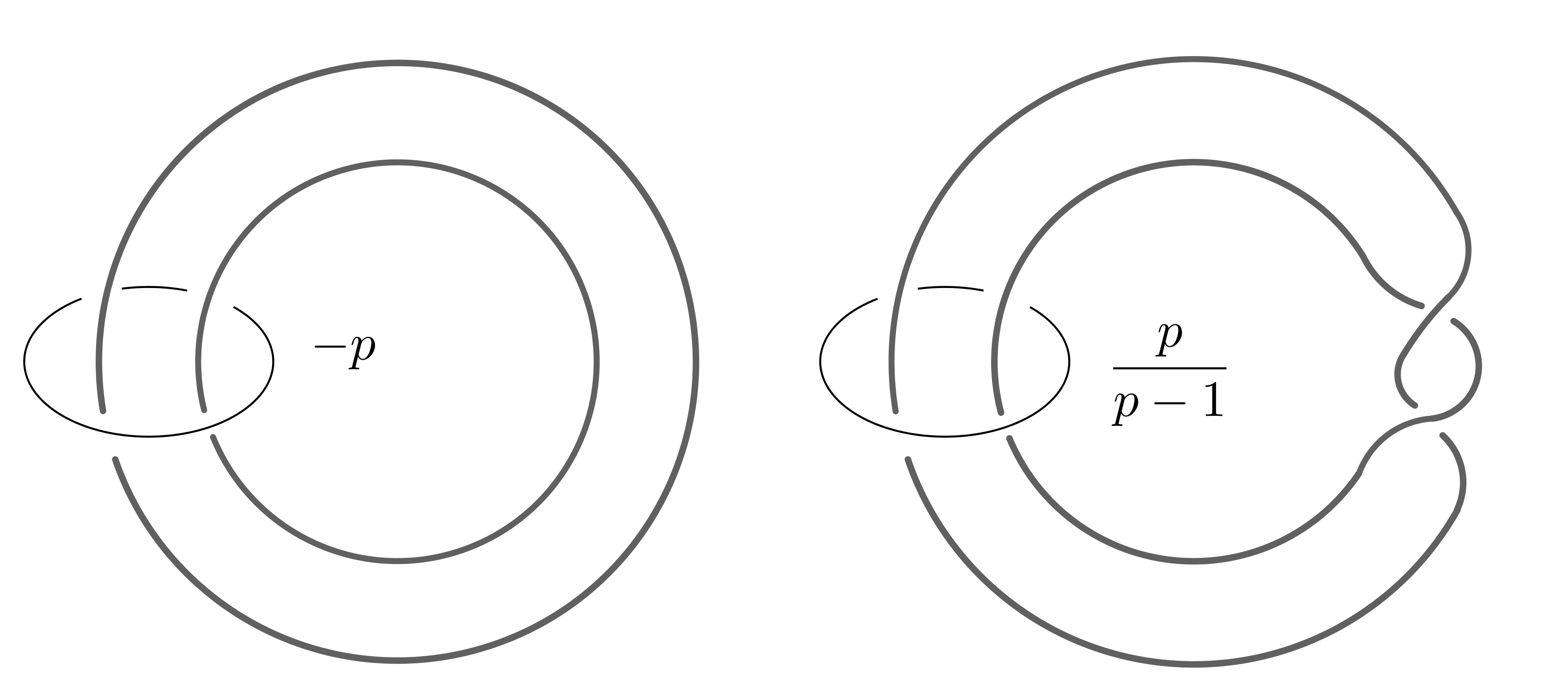}
	\caption{Changing framing of the fixed part: from $-p$ to $p/p-1$.}\label{rolfsen}
\end{figure}

Now we are ready to get the lift for $p>1$.  We start constructing a band diagram of $L$ by following the procedure described in Section 1.3 (see Figure~\ref{linklift}(1)). The band diagram results to be the braid $\Delta_2^{-2}=\sigma_1^{-2}$ (see Figure~\ref{linklift2}(1)). Since we have only two components, the Garnside braid coincides with the generator $\sigma_1$ of the braid group. If we now apply Proposition~\ref{lift}, we see that the lift of $L$ in $S^3$ is the closure of the braid
\[
\Lambda^-:=(\Delta_2^{-2})^p\Delta_2^{2(p-1)}=(\sigma_1^{-2})^p\sigma_1^{2(p-1)}=\sigma_1^{-2},
\]
which becomes the positive Hopf link $H^+$ with the orientation previously fixed (see Figure~\ref{linklift2}(2)).

The case $p<1$ results to be easier than the previous one, since the framing is already positive. We construct firstly the punctured disk diagram and then the associated band diagram (see Figure~\ref{linklift}(1)). If we now apply Proposition~\ref{lift}, the lift of $L$ in $S^3$ coincides with the closure of the braid
\[
\Lambda^+:=\Delta_2^{2q}=\sigma_1^2,
\]
which is precisely the negative Hopf band $H^-$, according to the given orientation (see Figure~\ref{linklift}(2)).
\begin{figure}[!h]
	\centering
		\includegraphics[width=13cm]{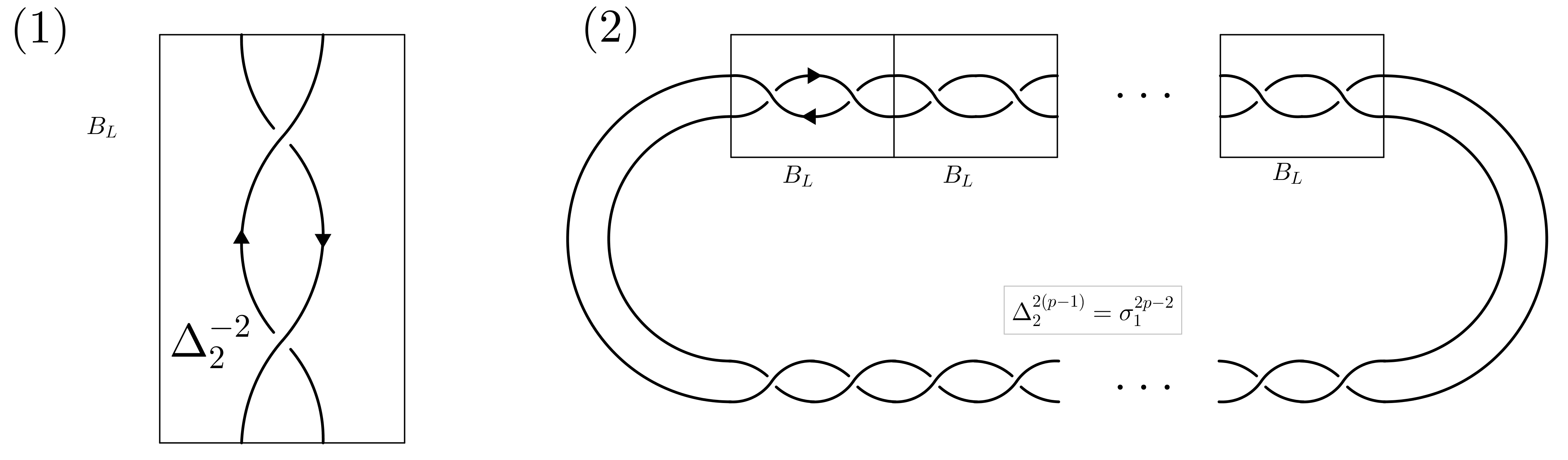}
	\caption{(1) Band diagram for $L$, (2) Lift for $p>1$.}\label{linklift2}
\end{figure}

\begin{figure}[!h]
	\centering
		\includegraphics[width=13cm]{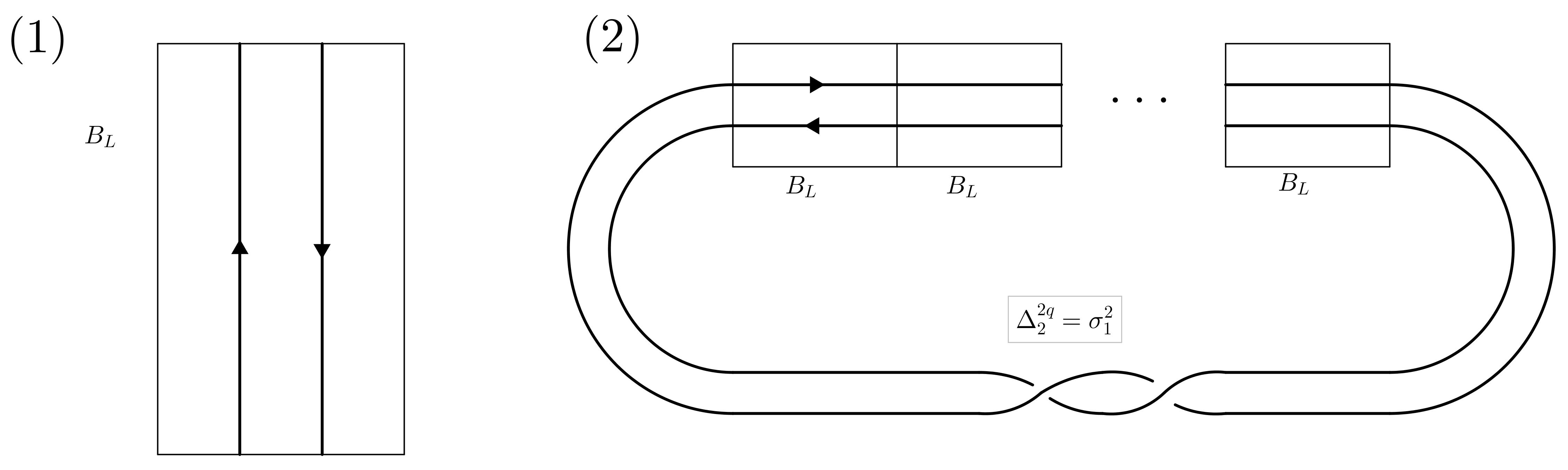}
	\caption{(1) Band diagram for $L$, (2) Lift for $p<1$.}\label{linklift}
\end{figure}

For both $p>1$ and $p<1$, the lift reveals to be a Hopf link, which is fibered in $S^3$. We are going to show that in $L(p,-1)$ the link of Proposition~\ref{casep1} can be obtained as the quotient of $H^+$ under the action of $\mathbb{Z}_p$. First of all, we remember that for $H^+:=\{ (z_1,z_2) \in S^3 | z_1z_2=0\}$ the projection map which realizes the fibration is given by
\[
\pi^+: S^3 \setminus H^+ \rightarrow S^1, \hspace{15pt} \pi_+(\rho_1,\theta_1,\rho_2,\theta_2):=\theta_1 + \theta_2.
\]

For the space $L(p,-1)$ the action of the generator $\zeta \in \mathbb{Z}_p$ can be written as follows
\[
\zeta \cdot (\rho_1e^{i\theta_1},\rho_2e^{i\theta_2})=(\rho_1e^{i(\theta_1+\frac{2\pi}{p})},\rho_2e^{i(\theta_2-\frac{2\pi}{p})})
\]
for every point $(\rho_1e^{i\theta_1},\rho_2e^{i\theta_2}) \in S^3$. Hence, the components of $H^+$ are stable under the $\mathbb{Z}_p$ action. In fact, choosing for instance the component $\{ z_1 = 0\}$, it results
\[
\zeta \cdot (0,e^{i\theta_2})=(0,e^{i(\theta_2-\frac{2\pi}{p})}),
\]
so $\zeta \cdot e^{i\theta_1}$ is still a point of $\{z_1=0\}$. The same is true for the other component $\{z_2=0\}$, but we have to change the sign in the action of $\zeta$. In the same way, the map $\pi^+$ is stable under the action of $\zeta$. Indeed, we have
\[
\pi^+(\zeta \cdot (\rho_1,\theta_1,\rho_2,\theta_2))=\pi^+(\rho_1,\theta_1+\frac{2\pi}{p},\rho_2,\theta_2-\frac{2\pi}{p})=\theta_1+\frac{2\pi}{p}+\theta_2-\frac{2\pi}{p}
\]
and thanks to cancellation the result coincides with $\pi^+(\rho_1,\theta_1,\rho_2,\theta_2)$. Hence, we have a well defined map
\[
\pi^+_{p,-1}:L(p,-1) \setminus H^+_{p,-1} \rightarrow S^1, \hspace{15pt} \pi^+_{p,-1}[\rho_1,\theta_1,\rho_2,\theta_2]:=\theta_1+\theta_2
\]
where $[\rho_1,\theta_1,\rho_2,\theta_2]$ stands for the equivalence class of the point $(\rho_1,\theta_1,\rho_2,\theta_2)$ in $L(p,-1)$ and, with the introduction of an abuse of notation, \mbox{$H^+_{p,-1}:=H^+/\mathbb{Z}_p$}. We now consider the annulus embedded in $S^3$ and parametrized by the equation
\[
i_+:[0,1] \times S^1 \rightarrow S^3, \hspace{15pt} i_+(t,\theta):=(\sqrt{1-t^2},-\theta,t,\theta)
\]
and we set $A^+:=\text{Im}i_+$. This annulus is $\mathbb{Z}_p$ stable, since it holds
\[
\zeta \cdot i_+(t,\theta)=\zeta \cdot (\sqrt{1-t^2},-\theta,t,\theta)=(\sqrt{1-t^2},-\theta+\frac{2\pi}{p},t,\theta-\frac{2\pi}{p})
\]
which is exactly a point of the form $i_+(t,\theta-\frac{2\pi}{p}) \in A^+$. Then, by indicating with $A^+_{p,-1}:=A^+/\mathbb{Z}_p$, we have shown that the pair $(H^+_{p,-1},\pi^+_{p,-1})$ is an open book decomposition for $L(p,-1)$ with page $A^+_{p,-1}$.\\
The next is to prove the compatibility of $(H^+_{p,-1},\pi^+_{p,-1})$ with a suitable contact structure on $L(p,-1)$. To determine the desired contact structure we can observe that the standard contact structure $\xi_{st}=\ker \alpha$ induces a contact structure on every lens space $L(p,q)$ with $p \neq 0$. Indeed, the action of an element $\zeta^n \in \mathbb{Z}_p \subset U(1)$ fixes the form $\alpha$, as shows the following computation 
\[
\zeta^n \cdot \alpha = \rho_1^2 d(\theta_1 +\frac{2\pi n}{p}) + \rho_2^2 d(\theta  + \frac{2\pi nq}{p})=\rho_1^2 d\theta_1  + \rho_2^2 d\theta_2=\alpha, \hspace{10pt} n \in \mathbb{Z}
\]

and this proves that we have a well defined contact 1-form on $L(p,q)$ induced by $\alpha$. 

\begin{deft}
The contact structure induced on $L(p,q)$ by the standard contact 1-form of $S^3$ under the action of $\mathbb{Z}_p$ is called \textit{standard contact structure} on the lens space $L(p,q)$. We indicate the 1-form associated to this structure with $\alpha_{p,q}$.
\end{deft}

In order to show the compatibility of $\alpha_{p,-1}$ with $(H^+_{p,-1},\pi_{p,-1})$, it suffices to prove the compatibility of the standard contact structure on $S^3$ with $(H^+,\pi^+)$.
For a fixed value $\omega = \theta_1+\theta_2$, the page $(\pi^+)^{-1}(\omega)$ is parametrized by
\[
i_+:[0,1] \times S^1 \rightarrow S^3, \hspace{15pt} i_+(t,\theta):=(\sqrt{1-t^2},\omega-\theta,t,\theta).
\]

By considering
\[
\alpha^+:=i_+^*(\rho_1,\theta_1,\rho_2,\theta_2)=(1-t^2)d(\omega-\theta)+t^2d\theta=(2t^2-1)d\theta
\]
and computing its differential $d\alpha^+=d[(2t^2-1)d\theta]=4tdt\wedge d\theta$, we obtain a positive volume form on the annulus. Now, if we consider the component of $H^+$ given by $\{ z_1 = 0\}$, this has $\partial/\partial \theta_2$ as tangent vector and $\alpha(\partial/\partial \theta_2)>0$. In the same way, the component given by $\{ z_2 = 0\}$ has tangent vector $\partial / \partial \theta_1$ and $\alpha(\partial / \partial \theta_1) >0$, so we have proved that $\xi_{st}$ is supported by $(H^+,\pi^+)$.
\\
Since we have obtained two distinct open book decomposition for $L(p,-1)$, the one determined in Proposition~\ref{casep1} and the one induced by the open book $(H^+,\pi^+)$, we wish to show now the these two open book decompositions are equivalent.
To do this, we start considering the abstract open book associated to $(H^+_{p,-1},\pi^+_{p,-1})$. The page coincides with the annulus $A=S^1 \times [0,1]$. If $h$ is the monodromy, the decomposition results to be $(A,h)$. Since $h$ in an element of the mapping class group $\text{MCG} (A)$, which is isomorphic to group $\mathbb{Z}$ generated by the positive Dehn twist along the core circle $\gamma=S^1 \times \{1/2\}$, we must have $h:=D^k_\gamma$, with $k \in \mathbb{Z}$. 
Thanks to Lemma~\ref{obdeffects}, we know that $k$ has to be equal to $p$ in order to get back the space $L(p,-1)$, so the abstract results to be $(A,D^p_\gamma)$. But this decomposition is exactly the one associated to the open book obtained in Proposition~\ref{casep1}. Thus the two decompositions obtained must be diffeomorphic.

%% file: materialefinale/Bibliografia.tex

\addcontentsline{toc}{chapter}{\bibname} 	

\vspace{40pt}
Enrico Manfredi\\
Department of Mathematics,\\
University of Bologna,\\
Piazza di Porta San Donato 5,\\
40126 Bologna,\\
Italy\\
enrico.manfredi3@unibo.it\\
\\
Alessio Savini\\
Department of Mathematics,\\
University of Bologna,\\
Piazza di Porta San Donato 5,\\
40126 Bologna,\\
Italy\\
alessio.savini5@unibo.it